\documentclass[a4paper,10pt]{article}

\usepackage[T1]{fontenc}
\usepackage{lmodern}
\usepackage[utf8]{inputenc}
\usepackage[english]{babel}
\usepackage{geometry}

\usepackage{booktabs}
\usepackage{color}
\usepackage{amsthm}
\usepackage{amsmath,amssymb}
\usepackage{graphicx}
\usepackage{listings}
\usepackage{emptypage}
\usepackage{newlfont}
\usepackage{verbatim}
\usepackage{amssymb}
\usepackage{array}
\usepackage{mathtools}
\usepackage{afterpage}
\usepackage[nottoc,notlot,notlof]{tocbibind}
\usepackage{csquotes}
\usepackage{tabularx}
\usepackage[nottoc]{tocbibind}

\newcommand{\numberset}{\mathbb}
\newcommand{\N}{\numberset{N}}

\newcommand{\R}{\numberset{R}}

\newcommand{\ud}{\mathrm{d}}

\theoremstyle{plain} 
\newtheorem{thm}{Theorem} 
\newtheorem{cor}[thm]{Corollary} 
\newtheorem{lem}[thm]{Lemma} 
\newtheorem{prop}[thm]{Proposition} 
\newtheorem*{theorem*}{Teorema}

\theoremstyle{definition} 
\newtheorem{defn}[thm]{Definition}

\newtheorem{rem}[thm]{Remark}

\usepackage{setspace}
\usepackage{fancyhdr}
\usepackage{pdfpages}
\usepackage{standalone}
\usepackage{appendix}

\numberwithin{equation}{section}
\numberwithin{thm}{section}

\title{\textbf{On the existence of minimal expansive solutions to the $N$-body problem}}
\author{Davide Polimeni and Susanna Terracini}
\date{}

\begin{document}

\maketitle

\begin{abstract}

We deal, for the classical $N$-body problem, with the existence of action minimizing half entire expansive solutions with prescribed asymptotic direction and initial configuration of the bodies. We tackle the cases of hyperbolic, hyperbolic-parabolic and parabolic arcs in a unitary manner.  Our approach is based on the minimization of a renormalized Lagrangian action, on a suitable functional space. With this new strategy, we are able to confirm the already-known results of the existence of both hyperbolic and parabolic solutions, and we prove for the first time the existence of hyperbolic-parabolic solutions for any prescribed asymptotic expansion in a suitable class. Associated with each element of this class we find a viscosity solution of the Hamilton-Jacobi equation as a linear correction of the value function. Besides, we also manage to give a better description of the growth of parabolic and hyperbolic-parabolic solutions. 

\end{abstract}

\section{Introduction and main results}\label{sec_intro}

In this paper, we deal with half entire solutions to the $N$-body problem of Celestial Mechanics in the Euclidean space $\R^d$ of hyperbolic, parabolic or mixed hyperbolic-parabolic type. 
We first investigate the existence of trajectories to the gravitational $N$-body problem having prescribed growth at infinity. This classical line  of research has recently been re-energized by the injection of new methods of analysis, of perturbative, variational, geometric and/or analytic functional nature. Indeed, in addition to the classical literature on the subject \cite{Alekseev_FinalMotions3Body, Chazy, MarchalSaari_FinalEvolution, Saari_ExpandingGravitational,Saari_ManifoldStructure}, we quote the recent results about existence of hyperbolic solutions \cite{MR4121133, MR3868425, MR4600219,MadernaVenturelli_GloballyMinimizingParabolic}, parabolic ones \cite{BTV2, BTV1,BDFT, LuzMaderna_FreeTimeMinimizers,MadernaVenturelli_HyperbolicMotions,SaariHulkower_TotalCollapse} and hyperbolic-parabolic ones \cite{Burgos_PartiallyHyperbolic}, without neglecting those ending in an  oscillatory manner \cite{MR3455155,MR4265664,paradela2022oscillatory} and references therein.

To start with, let us consider $N$ point masses $m_1,...,m_N > 0$ moving  under the action of the mutual attraction, with the inverse-square law of universal gravitation. We denote the components of the configuration vector $x = (r_1,...,r_N) \in \R^{dN}$ of the positions of the bodies and  by $|r_i - r_j|$ the Euclidean distance between two bodies $i$ and $j$. Newton’s equation of motion for the $i$-th body of the $N$-body problem reads as
\begin{displaymath}
    m_i\ddot{r_i} = -\sum_{j=1,...,N,\ j\neq i}^{N} m_im_j \frac{r_i - r_j}{|r_i - r_j|^3}.
\end{displaymath}
Since these equations are invariant by translation, we can fix the origin of our inertial frame at the center of mass of the system. We can thus define the configuration space of the system as
\begin{displaymath}
    \mathcal{X} = \biggl\{x=(r_1,...,r_N)\in\R^{dN},\ \sum_{i=1}^{N} m_i r_i = 0\biggl\}
\end{displaymath}
and denote by $\Omega = \{x\in \mathcal{X} \ | \   r_i\neq r_j\ \forall \ i \neq j\}  \subset \mathcal{X}$ the set of configurations without collisions, which is open and dense in $\mathcal{X}$, and with $\Delta$ its complement, that is the collision set.
Now we can write the equations of motion as
\begin{equation}\label{eq_newton}
    \mathcal{M}\ddot{x} = \nabla U(x),
\end{equation}
where $\mathcal{M}=\text{diag}(m_1 I_d,...,m_N I_d)$ is the matrix of the masses and the function $U:\Omega \rightarrow \mathbb{R} \cup \{+\infty\}$ is the Newtonian potential
\begin{equation}\label{eq:potential}
    U(x) = \sum_{i<j} \frac{m_i m_j}{|r_i-r_j|}.
\end{equation}
Newton's equations define an analytic local flow on $\Omega \times \R^{dN}$ with a first integral given by the mechanical energy:
\begin{displaymath}
    h = \frac{1}{2}\|\dot{x}\|_\mathcal{M}^2 - U(x).
\end{displaymath}
We will use $\|\cdot\|_\mathcal{M}$ to denote the norm induced by the mass scalar product
\begin{displaymath}
     \langle x,y\rangle_\mathcal{M} = \sum_{i=1}^{N} m_i \langle r_i,s_i\rangle,\qquad \text{for any }x=(r_1,...,r_N),\ y=(s_1,...,s_N)\in\mathcal{X},
\end{displaymath}
where, with a little abuse, $\langle\cdot,\cdot\rangle$ denoes the standard scalar product in $\R^d$ and also in $\mathcal{X}$.   

In this paper we will be concerned with the class of expansive motions, which is defined in the following way.
\begin{defn}
    A motion $x:[0,+\infty) \rightarrow \Omega$ is said to be expansive when all the mutual distances diverge, that is, when $|r_i(t)-r_j(t)| \rightarrow +\infty$ as $t\rightarrow+\infty$ for all $i < j$. Equivalently, the motion is expansive if $U(x(t)) \rightarrow 0$ as $t\rightarrow+\infty$.
\end{defn}
From the conservation of the energy, we observe that, since $U(x(t)) \rightarrow 0$ implies $\|\dot{x}(t)\|_\mathcal{M}^2\rightarrow 2h$ as $t\rightarrow+\infty$, expansive motions can only occur at nonnegative energies.

For a given motion, we introduce the minimum and the maximum separation between the bodies at time $t$ as the two functions
\begin{displaymath}
    r(t) = \min_{i<j} |r_i(t) - r_j(t)| \quad \text{and} \quad R(t) = \max_{i<j} |r_i(t) - r_j(t)|,
\end{displaymath}
where we write $|\cdot|$ to denote the standard Euclidean norm in $\R^d$.
The next fundamental theorems give us a more accurate description of the system's expansion.
\begin{thm}[Pollard, 1967 \cite{Pollard_BehaviorOfGravitationalSystems}]
    Let $x$ be a motion defined for all $t>t_0$. If $r$ is bounded away from zero, then we have that $R=O(t)$ as $t\rightarrow+\infty$. In addition, $R(t)/t\rightarrow+\infty$ if and only if $r(t)\rightarrow0$.
\end{thm}
\begin{thm}[Marchal-Saari, 1976 \cite{MarchalSaari_FinalEvolution}]\label{art2_thm_MS}
Let $x$ be a motion defined for all $t > t_0$. Then either $R(t)/t \rightarrow +\infty$ and $r(t) \rightarrow 0$, or there is a configuration $a \in \mathcal{X}$ such that $x(t) = at + O(t^{2/3})$. In particular, for superhyperbolic
motions (i.e. motions such that $\limsup_{t \rightarrow +\infty} R(t)/t = +\infty$) the quotient $R(t)/t$ diverges. 
\end{thm}
\begin{thm}[Marchal-Saari, 1976 \cite{MarchalSaari_FinalEvolution}]
Suppose that $x(t) = at + O(t^{2/3})$ for some $a \in \mathcal{X}$ and that the motion is expansive. Then, for each pair $i < j$ such that $a_i = a_j$, we have $|r_i(t)-r_j(t)| \approx\footnote{
Given positive functions $f$ and $g$, we write $f \approx g$ when there exist two positive constants $\alpha$ and $\beta$ such that $\alpha \leq \frac{f}{g} \leq \beta$.
} t^{2/3}$.
\end{thm}

Next, let us recall the well known Chazy classification of the expansive motions  for the  three-body problem (cfr. \cite{Chazy}), based on the asymptotic order of growth of the distances between the bodies. This prevents an expansive motion to be  superhyperbolic, so we can assume that it is of the form $x(t) = at + O(t^{2/3})$ for some limit $a \in \mathcal{X}$. Assuming that the center of mass of the system is at rest, Chazy classified these motions as follows:
\begin{itemize}
    \item \textit{Hyperbolic}: $a \in \Omega$ and $|r_i(t)-r_j(t)| \approx t$ for all $i < j$;
    \item \textit{Hyperbolic-parabolic}: $a \in \Delta$ but $a \neq 0$;
    \item \textit{Completely parabolic}: $a = 0$ and $|r_i(t)-r_j(t)| \approx t^{2/3}$ for all $i < j$.
\end{itemize}

The following definition is in order.
\begin{defn}
    A motion $x(t)$ is said to have limit shape when there is a time dependent similarity $S(t)$ of the space $\R^d$ such that $S(t)x(t)$ converges to some configuration $a \neq 0$.
\end{defn}
In our case, there is a diagonal action of $S(t)$, which means that $S(t)x=(S(t)r_1,...,S(t)r_N)$ for $x=(r_1,...,r_N)\in \mathcal{X}$. In particular, for the case of (half) hyperbolic motions, we can say that the limit shape of such a motion is its asymptotic velocity $a = \lim_{t\rightarrow +\infty} \frac{x(t)}{t}$. Similarily, (half) parabolic motions also possess a limit shape, which is now bound to be a central configuration, that is, a critical point of the potential $U$ constrained on the intertia ellipsoid $\mathcal E=\{x\in\mathcal X\,:\, \|x\|_\mathcal{M}^2=1\}$.  
\\

In this paper, we are going to tackle the existence of half entire expansive solutions for the Newtonian $N$-body problem from a unitary perspective by a global variational approach, using a suitable renormalized action functional, as the Lagrangian is not expected to be integrable on the half line. In particular, referring to Chazy's classification, we will show a proof of existence of motions for each one of the previous three classes of motions. As a first step, we shall revisit recent works by E. Maderna and A. Venturelli about the existence of half hyperbolic and parabolic trajectories from this new angle.
\begin{thm}[Maderna and Venturelli 2020, \cite{MadernaVenturelli_HyperbolicMotions}]\label{thm_hyperbolic}
    Given $d\in\N$, $d\geq2$, for the Newtonian $N$-body problem in $\R^{d}$ there is a hyperbolic motion $x:[1,+\infty)\rightarrow\mathcal{X}$ of the form
    \begin{displaymath}
        x(t) = at- \log (t) \nabla U(a) + o(1)\quad\text{as }t\rightarrow+\infty,
    \end{displaymath}
    for any initial configuration $x^0=x(1)\in\mathcal{X}$ and for any collisionless configuration $a\in\Omega$.
\end{thm}

As far as the parabolic case is concerned, in addition to providing an alternative proof,  we will be able to extend the result of Maderna and Venturelli \cite{MadernaVenturelli_GloballyMinimizingParabolic} by improving the estimate of the remainder as follows.

\begin{thm}\label{thm_parabolic}
    Given $d\in\N$, $d\geq2$, for the Newtonian $N$-body problem in $\R^{d}$ there is a parabolic solution $x:[1,+\infty)\rightarrow\mathcal{X}$ of the form
    \begin{equation}\label{eq:parabolic}
        x(t) = \beta b_m t^{2/3}+o(t^{1/3^+})\quad\text{as }t\rightarrow+\infty,
    \end{equation}
    for any initial configuration $x^0=x(1)\in\mathcal{X}$, for any minimal normalized central configuration $b_m$ and for $\beta = \sqrt[3]{\frac{9}{2}U(b_m)}$.
\end{thm}
Here, a minimal central configuration is a minimizer of the potential $U$ constrained to the intertia ellipsoid $\mathcal E=\{x\in\mathcal X\,:\, \|x\|_\mathcal{M}^2=1\}$. 
As said, the existence of hyperbolic and parabolic solutions for the Newtonian $N$-body problem has already been proved by Maderna and Venturelli in 2020 and 2009, respectively. In \cite{MadernaVenturelli_HyperbolicMotions}, the authors proved the existence of hyperbolic motions for any prescribed limit shape, any initial configuration of the bodies and any positive value of the energy. These solutions, whose actions are infinite, were found as the limits of locally converging subsequences in families of minimizing motions, where the existence of the approximate solutions are minimal geodesics of the Maupertuis'-Jacobi metric. More specifically, these solutions  were obtained as the limits of solutions of sequences of approximating two-point boundary value problems. To exclude collisions, both proofs in \cite{MadernaVenturelli_HyperbolicMotions} and \cite{MadernaVenturelli_GloballyMinimizingParabolic} invoke Marchal's Principle ensuring the absence of collisions for action-minimizing paths (Theorem \ref{thm_marchal}). There tarjectories are characteristic curves of a global viscosity solutions for the Hamilton-Jacobi equation $H(x, \nabla u) = h$. In such, these solutions are fixed points of the associated Lax-Oleinik semigroup. In \cite{MadernaVenturelli_GloballyMinimizingParabolic}, for any starting configuration they proved the existence of parabolic arcs asymptotic to any prescribed normalized minimal central configuration.

Compared to Maderna and Venturelli's articles, in this paper we show  alternative and simpler proofs for the existence of hyperbolic and parabolic solutions in a unitary framework, which is based on a straightforward application of the Direct Method of the Calculus of Variations to minimize the renormalized Lagrangian actions associated to the problem.  This approach has the advantage of allow us to complement the existence of parabolic arcs with their (almost exact) expansion \eqref{eq:parabolic}. 

Finally, after proving Theorems \ref{thm_hyperbolic} and \ref{thm_parabolic}, we will extend our approach to similarly prove the existence of hyperbolic-parabolic solutions for the $N$-body problem. In order to state our main result we need to introduce the \emph{$a$-cluster partition} associated with a collision asymptotic velocity $a\in\Delta\setminus\{0\}$, where clusters are the equivalence classes of the relation $i\sim j \Longleftrightarrow a_i-a_j=0$. Given a cluster $K$, we consider the associated partial potential $U_K$, where the sum in \eqref{eq:potential} is restricted to the cluster $K$. 
The $a$-clustered potential $U_a$  is the sum of all the cluster potentials of the partition. Now we can state our main theorem:

\begin{thm}\label{thm_partially_hyperbolic}
    Given $d\in\N$, $d\geq2$, for the Newtonian $N$-body problem in $\R^{d}$ there is a hyperbolic-parabolic motion $x:[1,+\infty)\rightarrow\mathcal{X}$ of the form
    \begin{displaymath}
        x(t) = at+\beta b_m t^{2/3}+o(t^{1/3^+})\quad\text{as }t\rightarrow+\infty,
    \end{displaymath}
    for any initial configuration $x^0=x(1)\in\mathcal{X}$, for any collision configuration $a\in\Delta$, for any normalized minimal central configuration\footnote{See Section \ref{sec_hyperbolic_parabolic} for the exact definition of $\beta$ and $b_m$.} $b_m\in\mathcal{X}$ of the $a$-clustered potential and for any choice of the energy constant $h>0$.
\end{thm}
Intuitively, hyperbolic-parabolic motions are those expansive motions of the form $x(t)=at+o(t)$, as $t\rightarrow+\infty$, when their limit shapes have collisions, that is, $a\in \Delta\setminus\{0\}$. This means that hyperbolic-parabolic motions can be viewed as clusters of bodies moving asymptotically with a linear growth, while the distances of the bodies inside each cluster grow with a rate of order $t^{2/3}$ and, referred to its center of mass, the cluster has a limit shape which is a prescribed minimal configuration of the cluster potential $U_K$. For the Newtonian $N$-body problem, the existence of hyperbolic-parabolic solutions for any prescribed positive energy and any given initial configuration of the bodies has been tackled by Burgos in \cite{Burgos_PartiallyHyperbolic}, where his proof follows from and application of  Maderna and Venturelli's Theorem on the existence of hyperbolic motions and a limiting procedure as the limit shape approaches the collision set. With respect to Burgos' result, we can provide a a much wider class of such hyperbolic-parabolic trajectories. Moreover, our approach provides a much more detailed information about the asymptotic behaviour of the solution and a better description of the motion of the bodies. Indeed, to prove Theorem \ref{thm_partially_hyperbolic}, we partition the set of bodies following the natural cluster partition that was presented by Burgos and Maderna in \cite{BurgosMaderna_GeodesicRays} and is defined as follows: if $x(t)=(r_1(t),...,r_N(t))$ and $a=(a_1,...,a_N)$, then $a_i=a_j$ if and only if $|r_i(t)-r_j(t)|=O(t^{2/3})$, and the partition of the set of bodies is defined by this equivalence relation.  Using this particular partition, we are able to decompose the Lagrangian action into two terms: the first is related to the hyperbolic motion of the clusters and the second is related to the parabolic motion of the bodies inside the clusters. Through similar proofs to the ones in Theorems \ref{thm_hyperbolic} and \ref{thm_parabolic}, we can thus apply the Direct Method of the Calculus of Variation and Marchal's Theorem also to the case of hyperbolic-parabolic motions.

\begin{cor}
    The motions $x(t)$ given by Theorems \ref{thm_hyperbolic}, \ref{thm_parabolic} and \ref{thm_partially_hyperbolic} are continuous at $t=1$ and collisionless for $t>1$. Moreover they are free time action minimizers at their energy level.
\end{cor}

As already pointed out by Maderna and Venturelli, a family of hyperbolic trajectories that are minimal in free time is associated, via the Busemann function, with a solution of the time-independent Hamilton-Jacobi equation. A further advantage of the approach through the direct minimization of a renormalized  action functional is that a value function, dependent on the initial point, is directly defined. As we shall outline in Section \ref{sec:HJ}, a linear correction to the value function is, as expected from theory, a solution of the Hamilton-Jacobi equation.

Our general strategy in the proofs of Theorems \ref{thm_hyperbolic}, \ref{thm_parabolic} and \ref{thm_partially_hyperbolic} is to seek solutions to \eqref{eq_newton} which are lower order perturbations of a given path:
\begin{displaymath}
    x(t) = r_0(t) + \varphi(t)+\Tilde x_0\,, \qquad \Tilde x_0=x^0-t_0(1).
\end{displaymath}
Here $\varphi(t)$ is the lower order term  and the reference path $r_0$ is linear in the hyperbolic case, is a parabolic self-similar solution in the parabolic one and mixes the two types in the hyperbolic-parabolic case and $\Tilde x_0=x_0-r_0(1)$. In particular, we will consider functions $\varphi$ belonging to the functional space of continuous functions on $[1,+\infty)$ which vanish at $t = 1$ and can be written as primitives of functions in $L^2(1,+\infty)$. With this choice of space, which is denoted by $\mathcal{D}_0^{1,2}(1,+\infty)$, we will be able to give the problem a global variational structure, so that we can prove the existence of solutions of the $N$-body problem through the minimization of a Lagrangian action on the space $\mathcal{D}_0^{1,2}(1,+\infty)$. The crucial idea will be to minimize the action after a necessary proper renormalization (cfr. Definition \ref{def:renormalized_action}), since the Lagrangian is never integrable at infinity.

\section{The variational setting}
                                                  
For the $N$-body problem, the Hamiltonian $H$ is defined over $\Omega\times\R^{dN}$ as
\begin{equation}\label{eq:hamiltonian}
    H(x,p) = \frac{1}{2}\|p\|_{\mathcal{M}^{-1}}^2 - U(x),
\end{equation}
while the Lagrangian is defined over $\Omega\times\R^{dN}$ as
\begin{displaymath}
    L(x,v) = \frac{1}{2}\|v\|_\mathcal{M}^2 + U(x).
\end{displaymath}
This means, in particular, that $L$ and $H$ become infinite when $x$ has collisions. Given two configurations $x,y\in\mathcal{X}$ and $T>0$, we denote by $\mathcal{C}(x,y,T)$ the set of absolutely continuous curves $\gamma:[a,b]\rightarrow\mathcal{X}$ going from $x$ to $y$ in time $T=b-a$ and we write $\mathcal{C}(x,y)=\bigcup_{T>0}\mathcal{C}(x,y,T)$. We define the Lagrangian action of a curve $\gamma\in\mathcal{C}(x,y,T)$ as the functional
\begin{displaymath}
    \mathcal{A}_L(\gamma) = \int_{a}^{b} L(\gamma,\dot{\gamma})\ \ud t = \int_{a}^{b} \frac{1}{2}\|\dot{\gamma}\|_\mathcal{M}^2 + U(\gamma)\ \ud t. 
\end{displaymath}

Hamilton's principle of least action implies that if a curve $\gamma$ is a minimizer of the Lagrangian action in $\mathcal{C}(x,y,T)$, then $\gamma$ satisfies Newton's equations at every time $t\in[a,b]$ in which $\gamma(t)$ has no collisions.  However, as Poincaré already noticed in \cite{Poincare_SolutionsPeriodiques}, there are curves with isolated collisions and finite action, which means that minimizing orbits may not always be true motions. The following theorem represents a big step forward in this theory, since it enabled the application of variational techniques to study the Newtonian $N$-body problem. The main idea to prove the theorem was given by Marchal in \cite{Marchal_MethodOfMinimization}, while more complete proofs are due to Chenciner in \cite{Chenciner} and Ferrario and Terracini in \cite{FerrarioTerracini}.
\begin{thm}[Marchal \cite{Marchal_MethodOfMinimization},  Chenciner  \cite{Chenciner}, Ferrario and Terracini \cite{FerrarioTerracini}]\label{thm_marchal}
Given $x,y\in\mathcal{X}$, if $\gamma\in\mathcal{C}(x,y)$ is defined on some interval $[a,b]$ and satisfies
\begin{displaymath}
    \mathcal{A}_L(\gamma) = \min\{\mathcal{A}(\sigma)\ |\ \sigma\in\mathcal{C}(x,y,b-a)\},
\end{displaymath}
then $\gamma(t) \in \Omega$ for all $t \in (a,b)$.
\end{thm}
Marchal's Theorem will be fundamental in our proofs, since it will guarantee that the minimizers of the action (whose existence is the object of our proofs) are in fact true motions of the $N$-body problem free of collisions. The Principle of Least Action, jointly with Theorem \ref{thm_marchal}, has been widely applied in the search for collisionless periodic solutions to the N-body problem (cfr. e.g. \cite{MR1610784,FerrarioTerracini}). However, we must now build a suitable variational framework for the search of expansive solutions.

Our minimization will take place on the functional space
\begin{displaymath}
    \mathcal{D}_0^{1,2}([1,+\infty),\mathcal{X}) = \{ \varphi\in H_{loc}^1([1,+\infty),\mathcal{X})\ :\ \varphi(1)=0\ \text{and}\ \int_{1}^{+\infty} \|\dot{\varphi}(t)\|_\mathcal{M}^2\ \ud t<+\infty \},
\end{displaymath}
which is endowed with the norm
\begin{displaymath}
    \|\varphi\|_{\mathcal{D}} = \bigg( \int_{1}^{+\infty} \|\dot{\varphi}(t)\|_\mathcal{M}^2\ \ud t \bigg)^{1/2}.
\end{displaymath}
\begin{rem}
Given a configuration $\varphi=(\varphi_1,...,\varphi_n)\in\mathcal{D}_0^{1,2}([1,+\infty),\mathcal{X})$, we will say that its components belong to the space $\mathcal{D}_0^{1,2}([1,+\infty),\mathbb{R}^{d})$ and the $\mathcal{D}_0^{1,2}$-norm of each component is 
\begin{displaymath}
     \|\varphi_i\|_\mathcal{D} = \bigg( \int_{1}^{+\infty} |\dot{\varphi}_i(t)|^2\ \ud t \bigg)^{1/2},
\end{displaymath}
for $i=1,...,N$. We will write $\mathcal{D}_0^{1,2}(1,+\infty)$ to denote both the spaces $\mathcal{D}_0^{1,2}([1,+\infty),\mathcal{X})$ and $\mathcal{D}_0^{1,2}([1,+\infty),\mathbb{R}^{d})$, since it will be trivial to distinguish them. 
\end{rem}

\begin{prop}[Cfr. Boscaggin-Dambrosio-Feltrin-Terracini, 2021 \cite{BDFT}]
    The space $\mathcal{D}_0^{1,2}(1,+\infty)$ is a Hilbert space containing the set $C_c^\infty(1,+\infty)$ as a dense subspace.
\end{prop}

We recall here the following paramount Hardy-type inequality, which will be used several times in the paper. It states that the space $\mathcal{D}_0^{1,2}(1,+\infty)$ is continuously embedded in a weighted $L^2$-space with measure $\ud t/t^2$.
\begin{prop}[\textit{Hardy inequality}, Cfr. Boscaggin-Dambrosio-Feltrin-Terracini, 2021 \cite{BDFT}]\label{dis_hardy}
    For every $\varphi\in\mathcal{D}_0^{1,2}(1,+\infty)$, it holds that 
    \begin{equation}
        \int_{1}^{+\infty} \frac{\|\varphi(t)\|_\mathcal{M}^2}{t^2}\ \ud t \leq 4 \int_{1}^{+\infty}\|\dot{\varphi}(t)\|_\mathcal{M}^2\ \ud t,
    \end{equation}
    and moreover
      \begin{equation}\label{dis_space_D012}
        \sup_{t\in[1,+\infty)}\frac{\|\varphi(t)\|_\mathcal{M}^2}{t-1} \leq  \int_{1}^{+\infty}\|\dot{\varphi}(t)\|_\mathcal{M}^2\ \ud t.
    \end{equation}
\end{prop}

In order to prove the existence of minima for the functional $\mathcal{A}$ on $\mathcal{D}_0^{1,2}(1,+\infty)$, we will properly renormalize the Lagrangian action and, after proving its coercivity and weak lower semicontinuity, we will apply the Direct Method of the Calculus of Variations. In particular, we will use the following renormalization.
\begin{defn}[Renormalized Lagrangian action]\label{def:renormalized_action}
    Given a motion $x(t)\in\Omega$ of the form $x(t) = \varphi(t) + r_0(t) + \Tilde{x}^0$, where $\varphi\in\mathcal{D}_0^{1,2}(1,+\infty)$, $r_0(t) = at + \beta bt^{2/3}$ for proper $a,\beta b\in\mathcal{X}$, and $\Tilde{x}^0\in\mathcal{X}$, then we can define the renormalized Lagrangian action
    \begin{equation}\label{eq:renorm_action}
        \mathcal{A}^{ren}(\varphi) = \int_{1}^{+\infty} \frac{1}{2}\|\dot{\varphi}(t)\|_\mathcal{M}^2 + U(\varphi(t) + r_0(t) + \tilde{x}^0) - U(r_0(t)) - \langle \mathcal{M}\ddot{r}_0(t),\varphi(t)\rangle\ \ud t.
    \end{equation}
\end{defn}

In order to shorten the notation, throughout the paper we will usually write $\mathcal{A}$ instead of $\mathcal{A}^{ren}$. 

To describe the asymptotic expansion of our motions, we will use the following theorem and lemma. The former, applies to the cases of hyperbolic and hyperbolic-parabolic motions, while the latter, which is typically known as \textit{Chazy's Lemma}, states that the set of initial conditions in the phase space that generate hyperbolic motions is an open set and that the map defined on this set that gives the asymptotic velocity in the future is continuous. 

\begin{thm}[Chazy, 1922  \cite{Chazy}]\label{art2_thm_chazy}
Let $x(t)$ be a motion with energy constant $h > 0$ and defined for all $t > t_0$. 
\begin{enumerate}
    \item [(i)] The limit
        \begin{displaymath}
            \lim_{t \rightarrow +\infty} R(t)r(t)^{-1} = L \in [1,+\infty]
        \end{displaymath}
    always exists.
    \item [(ii)] If $L < +\infty$, then there are a configuration $a \in \Omega$ and some function $P$, which is analytic in a neighbrhood of $(0,0)$, such that for every $t$ large enough, we have
    \begin{displaymath}
        x(t) = at - \log (t) \nabla U(a) + P(u,v),
    \end{displaymath}
    where $u = 1/t$ and $v = \log (t) /t$.
\end{enumerate}
\end{thm}

\begin{lem}[Maderna-Venturelli, 2020 \cite{MadernaVenturelli_HyperbolicMotions}]\label{art2_lem4.1}
Working on an Euclidean space $E$, which is endowed with an Euclidean norm $\|\cdot\|$, let $U:E^N\rightarrow\mathbb{R}\cup\{+\infty\}$ be a homogeneous potential of degree -1 of class $C^2$ on the open set $\Omega = \{ x\in E^N\ |\ U(x)<+\infty \}$. Let $x:[0,+\infty)\rightarrow\Omega$ be a given solution of $\ddot{x}=\nabla U(x)$ satisfying $x(t)=at+o(t)$ as $t\rightarrow+\infty$ with $a\in\Omega$. Then we have the following:
\begin{enumerate}
    \item The solution $x$ has asymptotic velocity $a$, meaning that
    \begin{displaymath}
         \lim_{t\rightarrow+\infty} \dot{x}(t)=a.
    \end{displaymath}
    \item (Chazy's continuity of the limit shape). Given $\varepsilon >0$, there are constants $t_1 > 0$ and $\delta > 0$ such that, for any maximal solution $y:[0,T)\rightarrow \Omega$ satisfying $\|y(0)-x(0)\| < \delta$ and $\|\dot{y}(0)-\dot{x}(0)\|<\delta$, we have
    \begin{itemize}
        \item $T=+\infty$, $\|y(t)-at\|<t\varepsilon$ for all $t>t_1$;
        \item there is $b\in\Omega$ with $\|b-a\| < \varepsilon$ for which $y(t)=bt+o(t).$
    \end{itemize}
\end{enumerate}
\end{lem}

\section{Existence of minimal half hyperbolic motions}

This section is devoted to the proof of Theorem \ref{thm_hyperbolic}. The class of hyperbolic motions has the following equivalent definition, also due to Chazy (see \cite{Chazy}).
\begin{defn}\label{def_hyperbolic}
    Hyperbolic motions are those motions such that each body has a different limit velocity vector, that is, $\dot{r}_i(t) \rightarrow a_i \in \R^d$, as $t \rightarrow +\infty$, and $a_i \neq a_j$ whenever $i \neq j$.
\end{defn}

We consider the differential system
\begin{equation}\label{newton_eq}
    \begin{cases}
    \mathcal{M}\ddot{x}=\nabla U(x)\\
    x(1)=x^0\\
    \lim_{t\rightarrow+\infty}\dot{x}(t)=a
    \end{cases},
\end{equation}
where $x^0\in\mathcal{X}$ and $a\in\Omega$. 

To prove the existence of hyperbolic motions to Newton's equations \eqref{newton_eq}, we will look for solutions having the form $x(t)=\varphi(t)+at+x^0-a$, where $\varphi:[1,+\infty)\rightarrow\mathcal{X}$ belongs to the space $\mathcal{D}_0^{1,2}(1,+\infty)$. We can thus equivalently study the system
\begin{equation}\label{newton_eq_phi}
    \begin{cases}
    \mathcal{M}\ddot{\varphi}=\nabla U(\varphi + x^0 - a + at)\\
    \varphi(1)=0\\
    \lim_{t\rightarrow+\infty}\dot{\varphi}(t)=0
    \end{cases}.
\end{equation}
Taking advantage of the problem's variational structure, we would be tempted to prove the existence of hyperbolic motions through the minimization of the Lagrangian action associated to the system \eqref{newton_eq_phi}, that is, the functional
\begin{equation}\label{lag_ac_not_renorm}
    \int_{1}^{+\infty} \frac{1}{2}\|\dot{\varphi}(t)\|_\mathcal{M}^2 + U(\varphi(t) + x^0 - a +at)\ \ud t,
\end{equation}
where 
\begin{displaymath}
    U(\varphi(t) + x^0 - a +at) = \sum_{i<j} \frac{m_i m_j}{|(\varphi_i(t) + x^0_i - a_i + a_i t) - (\varphi_j(t) + x^0_j - a_j + a_j t)|}.
\end{displaymath}

In attempting to work with the action functional as above, the major problem we encounter is that $U(\varphi(t) + x^0 - a +at)$ needs not to be integrable at infinity. Indeed, when $\varphi\in C_0^\infty([1,+\infty))$, $U(\varphi(t) + x^0 - a +at)$ decays as $\frac{1}{t}$ for $t\rightarrow+\infty$. To overcome this problem, as we can add arbitrary functions to the Lagrangian without changing the associated Euler-Lagrange equations, we can renormalize the action functional in order to have a finite integral in the following way:
\begin{displaymath}
    \mathcal{A}(\varphi) = \mathcal{A}^{ren}(\varphi) = \int_{1}^{+\infty} \frac{1}{2}\|\dot{\varphi}(t)\|_\mathcal{M}^2 + U(\varphi(t) + x^0 - a +at) - U(at)\ \ud t.
\end{displaymath}

\subsection{Coercivity}

In order to apply the Direct Method of the Calculus of Variations, we start by proving the coercivity of the functional, that is to say, that $\mathcal{A}(\varphi) \rightarrow+\infty$ as $\|\varphi\|_\mathcal{D}\rightarrow+\infty$. From now on, we will use the notations $\varphi_{ij}=\varphi_i - \varphi_j$, $x^0_{ij}=x^0_i - x^0_j$ and $a_{ij}=a_i-a_j$. We observe that the action can be equivalently written as
\begin{displaymath}
    \mathcal{A}(\varphi) = \int_{1}^{+\infty} \frac{1}{2}\sum_{i=1}^{N} m_i|\dot{\varphi}_i(t)|^2 + U(\varphi(t) + x^0 - a +at) - U(at)\ \ud t, 
\end{displaymath}
where
\begin{displaymath}\begin{split}
    U(\varphi(t) + x^0 - a +at) - U(at) &= \sum_{i<j} \bigg(\frac{m_i m_j}{|(\varphi_i(t) + x^0_i - a_i + a_i t) - (\varphi_j(t) + x^0_j - a_j + a_j t)|} - \frac{m_i m_j}{|a_i - a_j|t}\bigg) \\
    & = \sum_{i<j} \bigg(\frac{m_i m_j}{|\varphi_{ij}(t) + x^0_{ij} - a_{ij} + a_{ij} t|} - \frac{m_i m_j}{|a_{ij}|t}\bigg).
\end{split}\end{displaymath}
Since we are working in the space of configurations whose center of mass is null at every time, it can easily be proved that 
\begin{equation}\label{alternative_proof_kinetic_term}
     \sum_{i=1}^{N} m_i |\dot{\varphi}_i(t)|^2 = \frac{1}{M}\sum_{i<j} m_i m_j |\dot{\varphi}_i(t) - \dot{\varphi}_j(t)|^2,
\end{equation}
where $M=\sum_{i=1}^{N}m_i$. Indeed, we have
\begin{displaymath}\begin{split}
     \sum_{i<j} m_i m_j |\dot{\varphi}_i(t) - \dot{\varphi}_j(t)|^2 &= \frac{1}{2} \sum_{i,j} m_i m_j (|\dot{\varphi}_i(t)|^2 +|\dot{\varphi}_j(t)|^2 - 2\langle\dot{\varphi}_i(t),\dot{\varphi}_j(t)\rangle)\\
     & = \frac{1}{2}\bigg( M\sum_{i=1}^{N}m_i |\dot{\varphi}_i(t)|^2 + M\sum_{j=1}^{N}m_j |\dot{\varphi}_j(t)|^2 - 2 \langle\sum_{i=1}^{N}m_i \dot{\varphi}_i(t) , \sum_{j=1}^{N}m_j \dot{\varphi}_j(t) \rangle \bigg) \\
     & = \frac{1}{2}\bigg( M\sum_{i=1}^{N}m_i |\dot{\varphi}_i(t)|^2 + M\sum_{j=1}^{N}m_j |\dot{\varphi}_j(t)|^2 \bigg) \\
     & = M \sum_{i=1}^{N}m_i |\dot{\varphi}_i(t)|^2.
\end{split}\end{displaymath}
Using \eqref{alternative_proof_kinetic_term}, we can then write the Lagrangian action as
\begin{displaymath}
     \mathcal{A}(\varphi) = \int_{1}^{+\infty} \sum_{i<j} m_i m_j \bigg( \frac{|\dot{\varphi}_{ij}(t)|^2}{2M} + \frac{1}{|\varphi_{ij}(t) + x^0_{ij} - a_{ij} + a_{ij} t|} - \frac{1}{|a_{ij}|t} \bigg)\ \ud t.
\end{displaymath}
Since $\|\dot{\varphi}\|_{L^2}\rightarrow+\infty$ if and only if there is $i<j$ such that $\|\dot{\varphi}_i - \dot{\varphi}_j\|_{L^2}\rightarrow+\infty$, we can prove the coercivity of the action by proving the coercivity of each term $\mathcal{A}_{ij}$, where
\begin{displaymath}
     \mathcal{A}(\varphi) = \sum_{i<j} \mathcal{A}_{ij}(\varphi)
\end{displaymath}
and
\begin{displaymath}
    \mathcal{A}_{ij}(\varphi) = \int_{1}^{+\infty} m_i m_j \bigg( \frac{|\dot{\varphi}_{ij}(t) |^2}{2M} + \frac{1}{|\varphi_{ij}(t) + x^0_{ij} - a_{ij} + a_{ij} t|} - \frac{1}{|a_{ij}|t} \bigg)\ \ud t.
\end{displaymath}

Using the inequality 
\begin{equation}\label{2.3}
    |\varphi_i(t)|\leq\|\varphi_i\|_\mathcal{D}\sqrt{t},\qquad \text{for every }i=1,...,N,\ t\geq1\text{ and }\varphi_i\in\mathcal{D}_0^{1,2},
\end{equation}
which follows from \eqref{dis_space_D012}, we have
\begin{displaymath}
    U(\varphi(t) + x^0 - a +at) - U(at) \geq \sum_{i<j} \bigg(\frac{m_i m_j}{\|\varphi_{ij}\|_{\mathcal{D}}\sqrt{t} + |x^0_{ij} - a_{ij}| + |a_{ij}| t} - \frac{m_i m_j}{|a_{ij}|t}\bigg);
\end{displaymath}
We can then look for an upper bound for the integral
\begin{displaymath}
    \int_{1}^{+\infty} \bigg( \frac{1}{|a_{ij}|t} - \frac{1}{|a_{ij}|t + \|\varphi_{ij}\|_{\mathcal{D}}\sqrt{t} + |x^0_{ij}|} \bigg)\ \ud t.
\end{displaymath}
Using the change of variables $t=s^2$, we obtain
\begin{equation}\label{int_1_hyp}
    \frac{2}{|a_{ij}|} \int_{1}^{+\infty} \bigg( \frac{1}{s^2} - \frac{1}{s^2 + \frac{\|\varphi_{ij}\|_{\mathcal{D}}}{|a_{ij}|}s + \frac{|x^0_{ij}-a_{ij}|}{|a_{ij}|}}\bigg)s\ \ud s.
\end{equation}
Since
\begin{displaymath}
    \begin{split}
    s^2 + \frac{\|\varphi_{ij}\|_{\mathcal{D}}}{|a_{ij}|}s + \frac{|x^0_{ij}-a_{ij}|}{|a_{ij}|} & = \bigg( s + \frac{\|\varphi_{ij}\|_{\mathcal{D}}}{2|a_{ij}|} \bigg)^2 - \frac{\|\varphi_{ij}\|^2_{\mathcal{D}}}{4|a_{ij}|^2} + \frac{|x^0_{ij}-a_{ij}|}{|a_{ij}|}\\
    & =  \frac{\|\varphi_{ij}\|^2_{\mathcal{D}}}{4|a_{ij}|^2} \bigg[ \bigg( \frac{2|a_{ij}|s}{\|\varphi_{ij}\|_{\mathcal{D}}} + 1 \bigg)^2 -1 +\frac{4|x^0_{ij}-a_{ij}| |a_{ij}|}{\|\varphi_{ij}\|_{\mathcal{D}}^2} \bigg],
    \end{split}
\end{displaymath}
\eqref{int_1_hyp} is equal to
\begin{equation}\label{int_2_hyp}
    \frac{2}{|a_{ij}|}\frac{4|a_{ij}|^2}{\|\varphi_{ij}\|^2_{\mathcal{D}}} \int_{1}^{+\infty} \bigg[ \frac{1}{\bigg( \frac{2|a_{ij}|s}{\|\varphi_{ij}\|_{\mathcal{D}}} \bigg)^2} - \frac{1}{\bigg( \frac{2|a_{ij}|s}{\|\varphi_{ij}\|_{\mathcal{D}}} +1 \bigg)^2 - 1 + \frac{4|x^0_{ij}-a_{ij}| |a_{ij}|}{\|\varphi_{ij}\|^2_{\mathcal{D}}}} \bigg] s\ \ud s. 
\end{equation}
Changing variables again with $\tau = \frac{2|a_{ij}|s}{\|\varphi_{ij}\|_\mathcal{D}}$, we obtain that \eqref{int_2_hyp} is equal to
\begin{displaymath}
    \frac{2}{|a_{ij}|} \int_{\frac{2|a_{ij}|}{\|\varphi_{ij}\|_\mathcal{D}}}^{+\infty} \bigg[ \frac{1}{\tau^2} - \frac{1}{(\tau+1)^2 - 1 + \frac{4|x^0_{ij}-a_{ij}| |a_{ij}|}{\|\varphi_{ij}\|^2_\mathcal{D}}} \bigg] \tau\ \ud \tau.
\end{displaymath}
Since we are interested in large values of $\|\varphi_{ij}\|_\mathcal{D}$, we can suppose that there is some $\lambda<1$ such that $\frac{4|x^0_{ij}-a_{ij}| |a_{ij}|}{\|\varphi_{ij}\|^2_\mathcal{D}} \leq \lambda$. We then have
\begin{equation}\label{alternative_proof1}
    \frac{2}{|a_{ij}|} \int_{\frac{2|a_{ij}|}{\|\varphi_{ij}\|_\mathcal{D}}}^{+\infty} \bigg[ \frac{1}{\tau^2} - \frac{1}{(\tau+1)^2 - 1 + \frac{4|x^0_{ij}-a_{ij}| |a_{ij}|}{\|\varphi_{ij}\|^2_\mathcal{D}}} \bigg] \tau\ \ud \tau \leq \frac{2}{|a_{ij}|} \int_{\frac{2|a_{ij}|}{\|\varphi_{ij}\|_\mathcal{D}}}^{+\infty} \bigg[ \frac{1}{\tau^2} - \frac{1}{(\tau+1)^2 - 1 + \lambda} \bigg] \tau\ \ud \tau.
\end{equation}
The integrand of the last integral is a positive function. We observe that it is asymptotic to $\frac{1}{\tau}$ as $\tau\rightarrow0$ and  to $\frac{1}{\tau^2}$ as $\tau\rightarrow+\infty$. In particular, the integral exists at infinity, uniformly in $\lambda$. Taking $\|\varphi_{ij}\|_\mathcal{D}$ large enough, we can equivalently study the integral
\begin{displaymath}
    \int_{\varepsilon}^{+\infty} \bigg[ \frac{1}{\tau^2} - \frac{1}{(\tau+1)^2 - 1 + \lambda} \bigg] \tau\ \ud \tau,
\end{displaymath}
where $\varepsilon = \frac{2|a_{ij}|}{\|\varphi_{ij}\|_\mathcal{D}} < 1$. Since the integrand is asymptotic to $\frac{1}{\tau}$ as $\tau\rightarrow0$, it is equivalent to consider the sum of integrals
\begin{displaymath}
    \int_{\varepsilon}^{1} \frac{1}{\tau}\ \ud \tau + \int_{1}^{+\infty} \bigg[ \frac{1}{\tau^2} - \frac{1}{(\tau+1)^2 - 1 + \lambda} \bigg] \tau\ \ud \tau,
\end{displaymath}
where the second integral is constant (we will call it $C_1$) and does not depend on $\varepsilon$. We have
\begin{displaymath}
    \int_{\varepsilon}^{1} \frac{1}{\tau}\ \ud \tau + \int_{1}^{+\infty} \bigg[ \frac{1}{\tau^2} - \frac{1}{(\tau+1)^2 - 1 + \lambda} \bigg] \tau\ \ud \tau = \log \tau\bigg|_\varepsilon^1 + C_1 = -\log \varepsilon + C_1.
\end{displaymath}
Then, as $\|\varphi_{ij}\|_\mathcal{D}\rightarrow+\infty$, we know that the integral on the right-hand side of \eqref{alternative_proof1} behaves
like
\begin{displaymath}
    \frac{2}{|a_{ij}|}\bigg(-\log \frac{2|a_{ij}|}{\|\varphi_{ij}\|_\mathcal{D}} + C_1 \bigg) = \frac{2}{|a_{ij}|}\bigg(\log \|\varphi_{ij}\|_\mathcal{D} + C_1 - \log 2|a_{ij}| \bigg) = \frac{2}{|a_{ij}|}(\log \|\varphi_{ij}\|_\mathcal{D} + C_2 ),
\end{displaymath}
where $C_2 = C_1 -\log 2|a_{ij}|$.

We have thus proved that 
\begin{displaymath}
    \int_{1}^{+\infty} \bigg( \frac{1}{|a_{ij}|t} - \frac{1}{|a_{ij}|t + \|\varphi_{ij}\|_{\mathcal{D}}\sqrt{t} + |x^0_{ij}-a_{ij}|}\bigg)\ \ud t \leq \frac{2}{|a_{ij}|}(\log \|\varphi_{ij}\|_\mathcal{D} + C_2 ).
\end{displaymath}
This means that given $R>0$, when $\|\varphi_{ij}\|_\mathcal{D} \geq R$ for $R$ large  enough, we have
\begin{displaymath}
    \mathcal{A}_{ij}(\varphi) \geq m_i m_j \bigg[ \frac{\|\varphi_{ij}\|^2_\mathcal{D}}{2M} - \frac{2}{|a_{ij}|} (\log\|\varphi_{ij}\|_\mathcal{D} + C_2 )\bigg]
\end{displaymath}
and we can conclude that $\mathcal{A}_{ij}(\varphi) \rightarrow+\infty$ as $\|\varphi_{ij}\|_\mathcal{D}\rightarrow+\infty$.

\subsection{Weak lower semicontinuity}

Now, we prove that the functional $\mathcal{A}$ is weakly lower semicontinuous. Since the kinetic term $\frac{1}{2}\|\dot{\varphi}(t)\|_\mathcal{M}$ is convex, it is straightforward that the term $\int_{1}^{+\infty} \frac{1}{2}\|\dot{\varphi}(t)\|_\mathcal{M}^2\ \ud t$ is weakly lower semicontinuous. However it is worthwhile noticing that Fatou's Lemma cannot be applied to the term $\int_{1}^{+\infty} U(\varphi(t) + x^0 - a + at) - U(at)\ \ud t$, since the integrand is not a positive function, and we must proceed in a different way. We know that there is at least a sequence of functions in $\mathcal{D}_0^{1,2}(1,+\infty)$ that converges uniformly on the compact subsets of $[1,+\infty)$. To show this, consider a bounded sequence $(\varphi^n)_n$ in $\mathcal{D}_0^{1,2}(1,+\infty)$. We also know, by the definition of this space, that $\|\dot{\varphi}^n\|_{L^2([1,+\infty))}<+\infty$ and that $\varphi^n(1) = 0$, for every $n$. From the inequality 
\begin{equation}
    \|\varphi(t)\|_\mathcal{M}\leq \|\dot{\varphi}\|_{L^2}\sqrt{t-1}\leq\|\dot{\varphi}\|_{L^2}\sqrt{t}\qquad\text{ for every }t\geq1,
\end{equation}
we have $\|\varphi^n(t)\|_\mathcal{M} \leq \|\dot{\varphi}^n\|_{L^2}\sqrt{t}$ for every $t\geq1$ and for every $n$, which means that the $L^{\infty}$-norm in $[1,T]$ of $\varphi^n$ is bounded, for every fixed $T\geq1$ and for every $n$. On the other hand, we have
\begin{displaymath}
    \|\varphi^n(t_1)-\varphi^n(t_2)\|_\mathcal{M} \leq \|\dot{\varphi}^n\|_{L^2}\sqrt{t_1 - t_2},
\end{displaymath}
for every $t_1,t_2\in[1,+\infty)$ and for every $n$, which implies that the sequence $(\varphi^n)_n$ is equicontinuous on each interval $[1,T]$, for $T$ fixed. Then, by Ascoli-Arzelà's Theorem, we can say that for every fixed $T\geq1$ there is a subsequence $(\varphi^{n_k})_k$ that converges uniformly on $[1,T]$ (and, consequently, it converges pointwise on each compact). Besides, it can also be proved, through a diagonal procedure, that there is a subsequence converging pointwise in $[1,+\infty)$.

Consider now a sequence $(\varphi^n)_n$ in $\mathcal{D}_0^{1,2}(1,+\infty)$ converging weakly to some limit $\varphi\in\mathcal{D}_0^{1,2}(1,+\infty)$. By the properties of weak convergence we know that the sequence is bounded on $\mathcal{D}_0^{1,2}(1,+\infty)$ and, from the previous considerations, there is a subsequence $(\varphi^{n_k})_k$ converging uniformly on compact subsets of $[1,+\infty)$ (and hence pointwise in $[1,+\infty)$). We  write
\begin{equation}\label{alternative_proof_equality_with_s}
    \frac{1}{|x^0_{ij} - a_{ij} + a_{ij}t + \varphi^n_{ij}(t)|} - \frac{1}{|a_{ij}t|} = \int_{0}^{1} \frac{\ud}{\ud s}\bigg[\frac{1}{|a_{ij}t + s(x^0_{ij} - a_{ij} +\varphi^n_{ij}(t))|}\bigg]\ \ud s.
\end{equation}
However, this inequality holds only when the denominator of the integrand is not zero, which happens for $t$ sufficiently small. In particular, for all $s\in(0,1)$ we have
\begin{displaymath}\begin{split}
     |a_{ij}t + s(x^0_{ij} - a_{ij} + \varphi^n_{ij}(t))| & \geq |a_{ij}|t - s(|x^0_{ij} - a_{ij}| + \|\varphi^n_{ij}\|_\mathcal{D}\sqrt{t}) \\
     & > |a_{ij}|t - (|x^0_{ij} - a_{ij}| + \|\varphi^n_{ij}\|_\mathcal{D}\sqrt{t}),
\end{split}\end{displaymath}
and, since $|\varphi^n_{ij}(t)| \leq k\sqrt{t}$ for $k\in\mathbb{R}^+$ large  enough, we have
\begin{displaymath}
    |a_{ij}t + s(x^0_{ij} - a_{ij} + \varphi^n_{ij}(t))| > |a_{ij}|t - (|x^0_{ij} - a_{ij}| + k\sqrt{t}),
\end{displaymath}
where the last term is larger then zero if $t$ is larger than some $\bar{T}=\bar{T}(k)$; it is easy to compute $\bar{T}$ by studying the function $g(t)=|a_{ij}|t - [|x^0_{ij} - a_{ij}| + k\sqrt{t}]$. For these reasons, it is better to study the potential term separately on the two intervals $[1,\bar{T}]$ and $[\bar{T},+\infty)$.

We observe that $U(x^0 - a + at + \varphi)\in L^1([1,\bar{T}])$, since
\begin{displaymath}
     \frac{1}{|x^0_{ij} - a_{ij} + a_{ij}t + \varphi^n_{ij}(t)|} \leq \frac{1}{|x^0_{ij} - a_{ij}| - |a_{ij}|t - \|\varphi^n_{ij}\|_\mathcal{D}\sqrt{t}}.
\end{displaymath}
Besides, since $U$ is a positive function, we can use the pointwise convergence of the sequence and Fatou's Lemma to state that
\begin{displaymath}
     \int_{1}^{\bar{T}} \frac{1}{|x^0_{ij} - a_{ij} + a_{ij}t + \varphi_{ij}(t)|}\ \ud t \leq \liminf_{n\rightarrow+\infty} \int_{1}^{\bar{T}} \frac{1}{|x^0_{ij} - a_{ij} + a_{ij}t + \varphi^n_{ij}(t)|}\ \ud t.
\end{displaymath}

Now, knowing that the sequence $(\varphi^n)_n$ is bounded, we wish to prove that the term $U(\varphi^n(t) + x^0 - a + at) - U(at)$ converges in $L^1([\bar{T},+\infty))$. By using \eqref{alternative_proof_equality_with_s}, we can write
\begin{displaymath}\begin{split}
    \int_{\bar{T}}^{+\infty}& \frac{1}{|x^0_{ij} - a_{ij} + a_{ij}t + \varphi^n_{ij}(t)|} - \frac{1}{|a_{ij}t|}\ \ud t\\
    & = \int_{\bar{T}}^{+\infty} \bigg( \int_{0}^{1} -\frac{[a_{ij}t + s(x^0_{ij} - a_{ij} +\varphi^n_{ij}(t))](x^0_{ij}-a_{ij}+\varphi^n_{ij}(t))}{|a_{ij}t + s(x^0_{ij} - a_{ij} +\varphi^n_{ij}(t))|^3}\ \ud s \bigg)\ \ud t.
\end{split}\end{displaymath}
Our goal is to find an upper bound for the term
\begin{displaymath}
    \int_{\bar{T}}^{+\infty} \bigg|\frac{1}{|x^0_{ij} - a_{ij} + a_{ij}t + \varphi^n_{ij}(t)|} - \frac{1}{|a_{ij}t|}\bigg|\ \ud t.
\end{displaymath}
To find the upper bound, we will need the inequality
\begin{equation}\label{alternative_proof_inequality}
    \frac{|b+c|^2}{|b|^2 - |c|^2} \geq \frac{1}{3}, \qquad \text{for each }b,c\in\mathbb{R}^d \text{ such that }|b|\geq2|c|,
\end{equation}
which can easily be proved by elementary calculus. 
By \eqref{alternative_proof_inequality} and using the fact that $|x^0_{ij} - a_{ij}|+\|\varphi^n_{ij}\|_\mathcal{D}\sqrt{t} \leq k'\sqrt{t}$ for $k'\in\mathbb{R}^+$ large  enough, we thus have
\begin{displaymath}\begin{split}
    \int_{\bar{T}}^{+\infty} \bigg| \int_{0}^{1} & -\frac{[a_{ij}t + s(x^0_{ij} -a_{ij} +\varphi^n_{ij}(t))](x^0_{ij}-a_{ij}+\varphi^n_{ij}(t))}{|a_{ij}t + s(x^0_{ij} -a_{ij} +\varphi^n_{ij}(t))|^3}\ \ud s \bigg|\ \ud t   \\
    & \leq \int_{\bar{T}}^{+\infty} \bigg( \int_{0}^{1} \frac{|x^0_{ij}-a_{ij}+\varphi^n_{ij}(t)|}{|a_{ij}t + s(x^0_{ij} -a_{ij} +\varphi^n_{ij}(t))|^2}\ \ud s \bigg)\ \ud t  \\
    & \leq \int_{\bar{T}}^{+\infty} \bigg( \int_{0}^{1} 3\frac{|x^0_{ij}-a_{ij}|+\|\varphi^n_{ij}\|_\mathcal{D}\sqrt{t}}{|a_{ij}t|^2 - s|x^0_{ij} - a_{ij} +\|\varphi^n_{ij}\|_\mathcal{D}\sqrt{t}|^2}\ \ud s \bigg)\ \ud t  \\
    & \leq \int_{\bar{T}}^{+\infty} \bigg( \int_{0}^{1} \frac{3k'\sqrt{t}}{|a_{ij}|^2 t^2 - sk't}\ \ud s \bigg)\ \ud t.
\end{split}\end{displaymath}
By choosing $\bar{T}(k)\gg k'/|a_{ij}|^2$ so that $|a_{ij}|^2 t > sk'$ for all $s\in(0,1)$ and for all $t\in[\bar{T},+\infty)$ (take $k$ large  enough), we have that the last integral is finite and we have thus proved that there is a $\hat{T}$ such that, for all $\bar{T}\geq\hat{T}$,
\begin{displaymath}
     \int_{\bar{T}}^{+\infty} \bigg|\frac{1}{|x^0_{ij} - a_{ij} + a_{ij}t + \varphi^n_{ij}(t)|} - \frac{1}{|a_{ij}t|}\bigg|\ \ud t < +\infty.
\end{displaymath}
From this result, the $L^1$ convergence of the term $U(\varphi^n(t) + x^0 - a + at) - U(at)$ follows: by the dominated convergence Theorem we have, in particular,
\begin{displaymath}
     \lim_{n\rightarrow+\infty} \int_{\bar{T}}^{+\infty} U(\varphi^n(t) + x^0 - a + at) - U(at)\ \ud t = \int_{\bar{T}}^{+\infty} U(\varphi(t) + x^0 - a + at) - U(at)\ \ud t.
\end{displaymath}

Thus, if we consider any sequence $(\varphi^n)_n$ in $\mathcal{D}_0^{1,2}(1,+\infty)$ converging weakly to some $\varphi\in\mathcal{D}_0^{1,2}(1,+\infty)$, we have
\begin{displaymath}
     \mathcal{A}(\varphi) \leq \liminf_{n\rightarrow+\infty} \int_{1}^{+\infty} \frac{1}{2} \|\dot{\varphi}^n(t)\|_\mathcal{M}^2 + U(\varphi^n(t) + x^0 - a + at) - U(at)\ \ud t,
\end{displaymath}
which proves the weak lower semicontinuity of the renormalized  Lagrangian action in the space $\mathcal{D}_0^{1,2}(1,+\infty)$.

\begin{rem}
    The same reasoning leads to the continuity of the renormalised action with respect to the strong topology, in all elements $\varphi$ that do not give rise to collisions.
\end{rem}

\subsection{Absence of collisions and hyperbolicity of the motion}

Now we can apply the Direct Method of the Calculus of Variations, obtaining  a minimizer $\varphi$ on $\mathcal{D}_0^{1,2}(1,+\infty)$ of the renormalized action $\mathcal{A}$ and, by Marchal's Principle applied to $x(t) = \varphi(t) + x^0 - a + at$, we have that $x(t)\in\Omega$ for all $t\in(1,+\infty)$. Indeed, in each finite time interval, the full path $x$ minimizes the Lagrangian action among all paths joining the two ends.  Being free of collisions, it solves the associated Euler-Lagrange equations.

It remains to prove that $\lim_{t\rightarrow+\infty}\dot{\varphi}(t)=0$. We already know that $\dot{\varphi}\in L^2$ and that there is some $k\in\mathbb{R}^+$ such that $ \|\varphi(t)\|_\mathcal{M} \leq k\sqrt{t}$. By this last inequality, we have that
\begin{displaymath}
     \sum_{i<j} m_i m_j \frac{1}{|a_{ij}t + x^0_{ij} - a_{ij} + \varphi_{ij}(t)|} \leq \sum_{i<j} m_i m_j \frac{1}{|a_{ij}|t - |x^0_{ij}-a_{ij}| - k\sqrt{t}}
\end{displaymath}
and since $|a_{ij}|t - |x^0_{ij}-a_{ij}| - k\sqrt{t}\rightarrow+\infty$ as $t\rightarrow+\infty$ for all $i,j=1,...,N$, we obtain that $\lim_{t\rightarrow+\infty}U(x(t))=0$. Besides, since $\int_{1}^{+\infty} |\dot{\varphi}_{ij}(t)|^2\ \ud t < +\infty$, we have that
\begin{equation}\label{alternative_proof_liminf}
    \liminf_{t\rightarrow+\infty} |\dot{\varphi}_{ij}(t)|=0.
\end{equation}
\begin{rem}\label{rem_lim_hyp}
    A solution $x(t)=\varphi(t)+at+x^0-a$ of the equation $\mathcal{M}\ddot{x}=\nabla U(x)$ has positive energy. Indeed,
    \begin{displaymath}
         \frac{1}{2}\|\dot{x}(t)\|^2_\mathcal{M} - U(x(t)) = \frac{1}{2}\sum_{i=1}^{N} m_i |\dot{\varphi}_i(t) + a_i|^2 - U(x(t)) = h,
    \end{displaymath}
    and since by \eqref{alternative_proof_liminf} there is some $t_k\rightarrow+\infty$ such that $\lim_{t_k\rightarrow+\infty} \dot{\varphi}_i(t_k) = 0$, we have $h = \frac{1}{2}\|a\|_\mathcal{M}$.
\end{rem}
By Remark \ref{rem_lim_hyp}, we can apply Chazy's Lemma (Lemma \ref{art2_lem4.1}), which implies that the limit of $\dot{x}(t)$ exists for $t\rightarrow+\infty$. Since, by \eqref{alternative_proof_liminf}, there is at least a sequence $(t_k)_k$ such that $\dot{x}(t_k)\rightarrow a$ as $t_k\rightarrow+\infty$, we can conclude that \begin{displaymath}
    \lim_{t\rightarrow+\infty}\dot{x}(t)=a.
\end{displaymath}
Besides, we can apply Chazy's Theorem (Theorem \ref{art2_thm_chazy}) to state that the minimizing motion $x$ has the asymptotic expansion
\begin{displaymath}
    x(t) = at- \log (t) \nabla U(a) + o(1)\quad\text{as }t\rightarrow+\infty.
\end{displaymath}

We have thus proved that $x$ is a solution of the system
\begin{displaymath}
     \begin{cases}
     \mathcal{M}\ddot{x} = \nabla U(x)\\
     x(1) = x^0\\
     \lim_{t\rightarrow+\infty}\dot{x}(t) = a
     \end{cases},
\end{displaymath}
which means that there is a hyperbolic motion for the $N$-body problem, starting at any initial configuration $x^0$ and having prescribed asymptotic velocity $a$ without collisions.

\section{Existence of minimal half completely parabolic motions}\label{sec_parabolic}

We now focus on the class of completely parabolic motions, that is, those motions that have the form $x(t) = at + O(t^{2/3})$ for $t\rightarrow+\infty$, with $a=0$ and $|r_i(t)-r_j(t)|\approx t^{2/3}$ for $i<j$. Equivalently, we have the following definition.
\begin{defn}
    An expansive solution $x$ of the $N$-body problem is said to be parabolic if every body approaches infinity with zero velocity.
\end{defn}

In this section we will prove Theorem \ref{thm_parabolic}. More specifically, we will prove, for the $N$-body problem, the existence of orbits having the form 
\begin{displaymath}
    x(t) = \beta b t^{2/3} + o(t^{1/3^+}), \quad \text{as }t\rightarrow+\infty,
\end{displaymath}
where $\beta\in\R$ is a proper value and $b$ is a minimal central configuration. The remainder is $o(t^{1/3^+})$ in the sense that it grows less than order $\gamma$ for every $\gamma>1/3$.
\begin{defn}
We say that $b\in\mathcal{X}$ is a central configuration if it is a critical point of $U$ when restricted to the inertial ellipsoid
\begin{displaymath}
    \mathcal{E} = \{x\in\mathcal{X}\ :\ \langle \mathcal{M} x,x\rangle=1\}.
    \end{displaymath}
A central configuration $b_m\in \mathcal{E}$ is said to be minimal if
\begin{displaymath}
    U(b_m)=\min_{b\in\mathcal{E}} U(b).
\end{displaymath}
\end{defn}
More precisely, we will work with normalized central configurations, that is, central configurations $b$ such that $\langle \mathcal{M}b,b\rangle=1$.
\begin{rem}
Obviously, as $U$ is infinite on collisions,   minimal central configuration $b_m$ are non collision, i.e. $b_m\in\Omega$. 
\end{rem}

Given a Kepler potential $U$, we observe that from the definition of central configurations, it follows
\begin{displaymath}
    \nabla U(b) = \lambda \mathcal{M}b,
\end{displaymath}
where $\lambda$ is a Lagrange multiplier. Besides, we have the equality
\begin{equation}\label{parabolic_lambda}
    \lambda = \lambda\langle \mathcal{M}b,b\rangle = \langle\nabla U(b),b\rangle = -U(b).
\end{equation}

We first recall that there are self-similar solutions to Newton's equations $\mathcal{M}\ddot{x} = \nabla U(x)$ having the form
\begin{displaymath}
    x(t) = \beta b t^{2/3},
\end{displaymath}
for a proper constant $\beta$ and a central configuration $b$. Indeed
\begin{displaymath}
    \mathcal{M}\ddot{x} = -\frac{2}{9}\mathcal{M}\beta b t^{-4/3} = \nabla U(x) = \nabla U(\beta b t^{2/3}) = \frac{1}{\beta^2}t^{-4/3}\nabla U(b) = \frac{1}{\beta^2}t^{-4/3}\lambda \mathcal{M} b
\end{displaymath}
and, by \eqref{parabolic_lambda}, we also have
\begin{displaymath}
    \beta^3 = \frac{9}{2}U(b).
\end{displaymath}
This means that for $\beta = \sqrt[3]{\frac{9}{2}U(b)}$, the orbit $x(t)=\beta b t^{2/3}$ can be a homothetic solution of Newton's equations.

Now, let us define
\begin{displaymath}
    r_0(t) = \beta b_m t^{2/3},
\end{displaymath}
where $b_m\in\Omega$ is a normalized minimal central configuration. We wish to prove the existence of solutions of the system
\begin{displaymath}
    \begin{cases}
        \mathcal{M}\ddot{x} = \nabla U(x)\\
        x(1)=x^0\\
        \lim_{t\rightarrow+\infty}\dot{x}(t)=0
    \end{cases},
\end{displaymath}
given $x^0\in\mathcal{X}$. We seek solutions having the form
\begin{equation}\label{expr_parabolic_motion}
    x(t) = r_0(t) + \varphi(t) - r_0(1) - x^0 = r_0(t) + \varphi(t) + \Tilde{x}^0,
\end{equation}
where $\varphi\in\mathcal{D}_0^{1,2}(1,+\infty)$. In this case, we have 
\begin{displaymath}
    \nabla U(x(t)) = \mathcal{M}\ddot{x}(t) = \mathcal{M}\ddot{r_0}(t) + \mathcal{M}\ddot{\varphi}(t) = \nabla U(r_0(t)) + \mathcal{M}\ddot{\varphi}(t),
\end{displaymath}
which means that
\begin{displaymath}
    \mathcal{M}\ddot{\varphi}(t) = \nabla U(r_0(t)+\varphi(t)+\Tilde{x}^0)-\nabla U(r_0(t)).
\end{displaymath}
We can thus define the renormalized Lagrangian action as
\begin{equation}\label{parabolic_action}
    \mathcal{A}(\varphi) = \int_{1}^{+\infty} \frac{1}{2}\langle \mathcal{M}\dot{\varphi}(t),\dot{\varphi}(t)\rangle + U(r_0(t)+\varphi(t)+\Tilde{x}^0) - U(r_0(t)) - \langle \nabla U(r_0(t)),\varphi(t)\rangle\ \ud t.
\end{equation}
Besides the coercivity and weak lower semicontinuity of the Lagrangian action, we have to verify that:
\begin{itemize}
    \item $\forall\ \varphi\in\mathcal{D}_0^{1,2}(1,+\infty)$ such that $r_0(t)+\varphi(t)+\Tilde{x}^0(t)\neq0$ for all $t\geq1$, $\mathcal{A}(\varphi)<+\infty$;
    \item the action is continuous and $C^1$ on $\mathcal{D}_0^{1,2}\setminus\{\varphi\in\mathcal{D}_0^{1,2}\ :\ \exists\ t\text{ such that }r_0(t)+\varphi(t)+\Tilde{x}^0(t)=0\}$.
\end{itemize}

\subsection{Coercivity}

To minimize the action on the set $\mathcal{D}_0^{1,2}(1,+\infty)$, we start by proving its coercivity. We do this by reconducting the problem to a Kepler problem, where we denote $U_{min}=\min_{b\in\mathcal{E}}U(b)$. We notice that, for any orbit $x$,
\begin{displaymath}
    U(x) \geq \frac{U_{min}}{\|x\|},
\end{displaymath}
where $\|\cdot\|$ represents the Euclidean norm on $\R^{dN}$. Indeed, because of the homogeneity of the potential,
\begin{equation}\label{parabolic_inequality(a)}
    U(x) = U\bigg(\|x\| \frac{x}{\|x\|}\bigg) = \frac{1}{\|x\|}U\bigg(\frac{x}{\|x\|}\bigg) \geq \frac{1}{\|x\|}U_{min}.
\end{equation}
Besides,
\begin{equation}\label{parabolic_equality(b)}
    \nabla U(r_0) = \nabla U(\beta b_m t^{2/3}) = \frac{1}{\beta^2 t^{4/3}}\nabla U(b_m) = \frac{1}{\beta^2 t^{4/3}}\lambda \mathcal{M} b_m = -\frac{U_{min}}{\beta^2 t^{4/3}}\mathcal{M} b_m.
\end{equation}
Using \eqref{parabolic_inequality(a)} and \eqref{parabolic_equality(b)}, we can then write
\begin{displaymath}
    \begin{split}
        \mathcal{A}(\varphi) &\geq \int_{1}^{+\infty} \frac{1}{2}\langle \mathcal{M}\dot{\varphi}(t),\dot{\varphi}(t)\rangle + \frac{U_{min}}{\|r_0(t)+\varphi(t)+\Tilde{x}^0\|}-\frac{U_{min}}{\|r_0(t)\|} + \frac{1}{\beta^2 t^{4/3}}\langle U_{min}\mathcal{M} b,\varphi(t)\rangle\ \ud t \\
        & = \int_{1}^{+\infty} \frac{1}{2}\langle \mathcal{M}\dot{\varphi}(t),\dot{\varphi}(t)\rangle + \frac{U_{min}}{\|r_0(t)+\varphi(t)+\Tilde{x}^0\|}-\frac{U_{min}}{\|r_0(t)\|} + \frac{\langle U_{min}\mathcal{M}r_0(t),\varphi(t)\rangle}{\|r_0(t)\|^3}\ \ud t.
    \end{split}
\end{displaymath}

We have
\begin{displaymath}
    \|r_0(t)+\varphi(t)+\Tilde{x}^0\|^2 = \|r_0(t)\|^2 + 2\langle \mathcal{M}r_0(t),\varphi(t)\rangle + 2\langle \mathcal{M}\varphi(t),x^0\rangle + 2\langle \mathcal{M} r_0(t),x^0\rangle + \|\varphi(t)\|^2 + \|x^0\|^2 = u + v,
\end{displaymath}
where we define
\begin{displaymath}
    \begin{split}
        & u:= \|r_0(t)\|^2\\
        & v:= 2\langle \mathcal{M}r_0(t),\varphi(t)\rangle + 2\langle \mathcal{M}\varphi(t),x^0\rangle + 2\langle \mathcal{M} r_0(t),x^0\rangle + \|\varphi(t)\|^2 + \|x^0\|^2.
    \end{split}
\end{displaymath}
\begin{rem}\label{rem_taylor}
    The following equalities hold true:
    \begin{displaymath}
        \begin{split}
            & U(b+s)-U(b) = \int_{0}^{1} \frac{\ud}{\ud t} U(b+st)\ \ud t = \int_{0}^{1} \nabla U(b+st)\ \ud t,\\
            & U(b+s)-U(b)-\nabla U(b)s = \int_{0}^{1}\int_{0}^{1} \langle\nabla^2 U(b+st_1 t_2)s,s\rangle t_2\ \ud t_1\ \ud t_2.
        \end{split}
    \end{displaymath}
\end{rem}
Using Remark \ref{rem_taylor}, we then have
\begin{displaymath}
    \|r_0(t)+\varphi(t)+\Tilde{x}^0\|^{-1} = (u+v)^{-1/2} = u^{-1/2} -\frac{1}{2}u^{-3/2}v + \frac{3}{4}\int_{0}^{1}\int_{0}^{1}\langle(u+stv)^{-5/2}v,v\rangle s\ \ud s\ \ud t.
\end{displaymath}
Since the integral in the last expression is positive, it follows
\begin{equation}\label{parabolic_inequality3}
    \begin{split}
    \|r_0(t)+\varphi(t)+\Tilde{x}^0\|^{-1}  &= (u+v)^{-1/2}\\
    &\geq u^{-1/2} -\frac{1}{2}u^{-3/2}v \\
    &= \|r_0(t)\|^{-1} -\frac{1}{2\|r_0(t)\|^3}[2\langle \mathcal{M}r_0(t),\varphi(t)\rangle + 2\langle \mathcal{M}\varphi(t),x^0\rangle + 2\langle \mathcal{M} r_0(t),x^0\rangle + \|\varphi(t)\|^2 + \|x^0\|^2]\\
    & = \|r_0(t)\|^{-1} -\frac{\langle \mathcal{M}r_0,\varphi(t)\rangle}{\|r_0(t)\|^3} - \frac{\langle \mathcal{M}\varphi(t),x^0\rangle}{\|r_0(t)\|^3} - \frac{\langle \mathcal{M} r_0(t),x^0\rangle}{\|r_0(t)\|^3} -\frac{1}{2}\frac{\|\varphi(t)\|^2}{\|r_0(t)\|^3} - \frac{1}{2}\frac{\|x^0\|^2}{\|r_0(t)\|^3}.
    \end{split}    
\end{equation}
At this point we can use \eqref{parabolic_inequality3} to obtain
\begin{displaymath}
\begin{split}
    \mathcal{A}(\varphi) &\geq \int_{1}^{+\infty} \frac{1}{2}\langle \mathcal{M}\dot{\varphi}(t),\dot{\varphi}(t)\rangle + \frac{U_{min}}{\|r_0(t)+\varphi+\Tilde{x}^0\|}-\frac{U_{min}}{\|r_0(t)\|} + \frac{\langle U_{min}\mathcal{M} r_0(t),\varphi(t)\rangle}{\|r_0(t)\|^3}\ \ud t  \\
    & \geq \int_{1}^{+\infty} \frac{1}{2}\langle \mathcal{M}\dot{\varphi}(t),\dot{\varphi}(t)\rangle - \frac{U_{min}}{2}\frac{\|\varphi(t)\|^2}{\|r_0(t)\|^3}-\frac{\langle U_{min}\mathcal{M} \varphi(t),x^0\rangle}{\|r_0(t)\|^3}\ \ud t + C_3,
    \end{split}
\end{displaymath}
where $C_3$ is a constant. By Hardy inequality \eqref{dis_hardy} and the fact that, for $\beta=\sqrt[3]{\frac{9}{2}U(b)}$,
\begin{equation}\label{parabolic_equality4}
    \frac{U_{min}}{\|r_0(t)\|^3} = \frac{U_{min}}{\|\beta b_m t^{2/3}\|^3}=\frac{U_{min}}{\beta^3 t^2\| b_m\|^3}=\frac{2}{9}\frac{1}{t^2},
\end{equation}
we have
\begin{displaymath}
    \begin{split}
    \mathcal{A}(\varphi) &\geq \int_{1}^{+\infty} \frac{1}{2}\bigg[\langle \mathcal{M}\dot{\varphi}(t),\dot{\varphi}(t)\rangle -\frac{8}{9}\langle \mathcal{M}\dot{\varphi}(t),\dot{\varphi}(t)\rangle  \bigg] - \frac{U_{min}\langle \mathcal{M}\varphi(t),x^0\rangle}{\|r_0(t)\|^3}\ \ud t\\
    &=\int_{1}^{+\infty}\frac{1}{18} \langle \mathcal{M}\dot{\varphi}(t),\dot{\varphi}(t)\rangle - \frac{U_{min}\langle \mathcal{M}\varphi(t),x^0\rangle}{\|r_0(t)\|^3}\ \ud t.
    \end{split}
\end{displaymath}
Using again \eqref{parabolic_equality4}, we observe that
\begin{displaymath}
    \frac{U_{min}\langle \mathcal{M}\varphi(t),x^0\rangle}{\|r_0(t)\|^3} = \frac{2}{9}\frac{\langle \mathcal{M}\varphi(t),x^0\rangle}{t^2}.
\end{displaymath}
By Cauchy-Scwartz and Hardy inequalities, it follows
\begin{displaymath}
    \begin{split}
    \int_{1}^{+\infty} - \frac{U_{min}\langle \mathcal{M}\varphi(t),x^0\rangle}{\|r_0(t)\|^3}\ \ud t & \geq - \int_{1}^{+\infty} \frac{2}{9}\frac{|\langle \mathcal{M}\varphi(t),x^0\rangle|}{t^2}\ \ud t \geq - \int_{1}^{+\infty} \frac{2}{9}\frac{\|\varphi(t)\|_\mathcal{M}}{t}\frac{\|x^0\|_\mathcal{M}}{t}\ \ud t \\
    & \geq -\frac{2}{9}\bigg(\int_{1}^{+\infty}\frac{\|\varphi(t)\|_\mathcal{M}^2}{t^2}\ \ud t\bigg)^{1/2}\bigg(\int_{1}^{+\infty}\frac{\|x^0\|_\mathcal{M}^2}{t^2}\ \ud t\bigg)^{1/2}\ \ud t \\
    &\geq -\frac{4}{9}C_4\|\varphi\|_\mathcal{D},
    \end{split}
\end{displaymath}
where $C_4$ is constant. This means that
\begin{displaymath}
    \mathcal{A}(\varphi) \geq  \frac{1}{18}\|\varphi\|_\mathcal{D}^2-\frac{4}{9}C_4\|\varphi\|_\mathcal{D},
\end{displaymath}
which proves the coercivity of the action.

\subsection{Weak-lower semicontinuity}

Now, we can focus on the proof of the weak lower semicontinuity of the action. Consider a sequence of functions $(\varphi^n)_n\subset\mathcal{D}_0^{1,2}(1,+\infty)$ converging weakly in $\mathcal{D}_0^{1,2}(1,+\infty)$ to some $\varphi$, for $n\rightarrow+\infty$. It trivially follows that, for every $n$, $\|\varphi\|_\mathcal{D}<+\infty$ and $\|\varphi^n\|_\mathcal{D}<+\infty$. Let us divide the action in two parts:
\begin{displaymath}
    \mathcal{A}(\varphi) = \mathcal{A}_{[1,\overline{T})}(\varphi) + \mathcal{A}_{[\overline{T},+\infty)}(\varphi),
\end{displaymath}
where
\begin{displaymath}
    \begin{split}
        \mathcal{A}_{[1,\overline{T})}(\varphi)& = \int_{1}^{\overline{T}} \frac{1}{2}\|\dot{\varphi}(t)\|_\mathcal{M}^2 + U(r_0(t)+\varphi(t)+\Tilde{x}^0) - U(r_0(t)) - \langle \nabla U(r_0(t)),\varphi(t)\rangle \ud t,\\
        \mathcal{A}_{[\overline{T},+\infty)}(\varphi) &= \int_{\overline{T}}^{+\infty} \frac{1}{2}\|\dot{\varphi}(t)\|_\mathcal{M}^2 + U(r_0(t)+\varphi(t)+\Tilde{x}^0) - U(r_0(t)) - \langle \nabla U(r_0(t)),\varphi(t)\rangle\ \ud t
    \end{split}
\end{displaymath}
for some $\overline{T}\in(1,+\infty)$. Using Ascoli-Arzelà's Theorem, we can say that $\varphi^n\rightarrow\varphi$ uniformly on compact sets, which implies that $\langle \nabla U(r_0),\varphi^n\rangle \rightarrow \langle \nabla U(r_0),\varphi\rangle$ uniformly in $[1,\overline{T}]$, as $n\rightarrow+\infty$, for every $\overline{T}<+\infty$. Then, using Fatou's Lemma, it easily follows that the term $\mathcal{A}_{[1,\overline{T})}(\varphi)$ is weak lower semicontinuous.

Concerning the term $\mathcal{A}_{[\overline{T},+\infty)}(\varphi)$, we can write:
\begin{displaymath}
    \begin{split}
    \mathcal{A}_{[\overline{T},+\infty)}(\varphi) = \int_{\overline{T}}^{+\infty} &\frac{1}{2}\|\dot{\varphi}(t)\|_\mathcal{M}^2 + \frac{1}{2}\langle \nabla^2 U(r_0(t))\varphi(t),\varphi(t)\rangle\\
    &+ U(r_0(t)+\varphi(t)+\Tilde{x}^0) - U(r_0(t)) - \langle \nabla U(r_0(t)),\varphi(t)\rangle - \frac{1}{2}\langle \nabla^2 U(r_0(t))\varphi(t),\varphi(t)\rangle\ \ud t.
    \end{split}
\end{displaymath}
\textbf{Claim}: The map $\varphi(t) \mapsto \bigg( \int_{1}^{+\infty} \frac{1}{2}\|\dot{\varphi}(t)\|_\mathcal{M}^2 + \frac{1}{2}\langle \nabla^2 U(r_0(t))\varphi(t),\varphi(t)\rangle\ \ud t \bigg)^{1/2}$ is an equivalent norm to $\|\cdot\|_\mathcal{D}$. Indeed:
\begin{itemize}
    \item Since $U(x)\geq \frac{U_{min}}{\|x\|}$ for each $x\neq0$, it follows that $\nabla^2 U(x) \geq -U_{min}\frac{Id}{\|x\|^3}$, which implies $\langle \nabla^2 U(r_0(t))\varphi(t),\varphi(t)\rangle \geq -\frac{2}{9}\frac{\|\varphi(t)\|_\mathcal{M}^2}{t^2}$ for each $t\in[1,+\infty)$. Then, by Hardy inequality, we have
    \begin{displaymath}
        \int_{1}^{+\infty} \frac{1}{2}\|\dot{\varphi}(t)\|_\mathcal{M}^2 + \frac{1}{2}\langle \nabla^2 U(r_0(t))\varphi(t),\varphi(t)\rangle\ \ud t \geq \bigg( 1-\frac{8}{9} \bigg)\|\varphi\|^2_\mathcal{D} = \frac{1}{9}\|\varphi\|^2_\mathcal{D}.
    \end{displaymath}
    \item Using the fact that, for some constant $C_5>0$,
    \begin{displaymath}
        \langle \nabla^2 U(r_0(t))\varphi(t),\varphi(t)\rangle \leq C_5\frac{\|\varphi(t)\|_\mathcal{M}}{t^2}
    \end{displaymath}
    and Hardy inequality, we have
    \begin{displaymath}
        \int_{1}^{+\infty} \frac{1}{2}\|\dot{\varphi}(t)\|_\mathcal{M}^2 + \frac{1}{2}\langle \nabla^2 U(r_0(t))\varphi(t),\varphi(t)\rangle\ \ud t \leq C_6\|\varphi\|_\mathcal{D}^2,
    \end{displaymath}
    for some constant $C_6>0$.
\end{itemize}
From the equivalence between the two norms, we have that the term $\int_{\overline{T}}^{+\infty} \frac{1}{2}\|\dot{\varphi}(t)\|_\mathcal{M}^2 + \frac{1}{2}\langle \nabla^2 U(r_0(t))\varphi(t),\varphi(t)\rangle\ \ud t$ is weak lower semicontinuous. 

Using Taylor's series expansion, we can write
\begin{displaymath}
    \begin{split}
    \int_{\overline{T}}^{+\infty} & U(r_0(t)+\varphi(t)+\Tilde{x}^0) - U(r_0(t)) - \langle \nabla U(r_0(t)),\varphi(t)\rangle - \frac{1}{2}\langle \nabla^2 U(r_0(t)),\varphi(t),\varphi(t)\rangle\ \ud t \\
    & = \int_{\overline{T}}^{+\infty}\int_{0}^{1}\int_{0}^{1}\int_{0}^{1} \langle \nabla^3 U(r_0(t) + \tau_1\tau_2\tau_3 (\varphi^n(t)+\Tilde{x}^0))(\varphi^n(t)+\Tilde{x}^0),\varphi^n(t)+\Tilde{x}^0,\varphi^n(t)+\Tilde{x}^0\rangle \tau_1 \tau_2^2\ \ud \tau_1\ \ud \tau_2\ \ud \tau_3\ \ud t.
     \end{split}
\end{displaymath}
Obviously there is a $\Tilde{t}>1$ such that
\begin{displaymath}
    \|r_0(t) + \tau_1\tau_2\tau_3 (\varphi^n(t)+\Tilde{x}^0)\|_\mathcal{M}>0
\end{displaymath}
for every $t\geq\Tilde{t}$. We can then choose $\overline{T}\geq\Tilde{t}$ and we have 
\begin{displaymath}
    \langle \nabla^3 U(r_0(t) + \tau_1\tau_2\tau_3 (\varphi^n(t)+\Tilde{x}^0))(\varphi^n(t)+\Tilde{x}^0),\varphi^n(t)+\Tilde{x}^0,\varphi^n(t)+\Tilde{x}^0\rangle \leq C_7\frac{\|\varphi^n(t) + \Tilde{x}^0\|_\mathcal{M}^3}{t^{8/3}} \leq C_8\frac{\|\varphi^n\|_\mathcal{D}^3 t^{3/2}}{t^{8/3}} \leq \frac{C_9}{t^{7/6}},
\end{displaymath}
for every $t\geq\overline{T}$ and for proper constants $C_7,C_8,C_9>0$. This means that the term $\langle \nabla^3 U(r_0(t) + \tau_1\tau_2\tau_3 (\varphi^n(t)+\Tilde{x}^0))(\varphi^n(t)+\Tilde{x}^0),\varphi^n(t)+\Tilde{x}^0,\varphi^n(t)+\Tilde{x}^0\rangle \tau_1 \tau_2^2$ is $L^1$-dominated and the weak lower semicontinuity of $\mathcal{A}_{[\overline{T},+\infty)}$ follows from the Dominated Convergence Theorem.

\subsection{The renormalized action is of class $C^1$ over non-collision sets}

Now, we prove that the action $\mathcal{A}$ is $C^1$ over the set $\mathcal{D}_0^{1,2}([1,+\infty))\setminus\{\varphi\in\mathcal{D}_0^{1,2}\ : \exists\ t\text{ such that }r_0(t)+\varphi(t)+\Tilde{x}^0=0 \}$. The term $\int_{1}^{+\infty}\frac{1}{2}\langle \mathcal{M}\dot{\varphi}(t),\dot{\varphi}(t)\rangle\ \ud t = \frac{1}{2}\|\varphi\|^2_\mathcal{D}$ is of course a smooth functional, so we focus on the term
\begin{displaymath}
    \mathcal{A}^2(\varphi) := \int_{1}^{+\infty} K(t,\varphi(t))\ \ud t,
\end{displaymath}
where
\begin{displaymath}
    K(t,\varphi(t)):=U(r_0(t)+\varphi(t)+\Tilde{x}^0) - U(r_0(t)) - \langle \nabla U(r_0(t)),\varphi(t)\rangle.
\end{displaymath}
We have
\begin{displaymath}
    \ud \mathcal{A}^2(\varphi)[\psi] = \int_{1}^{+\infty} \langle\nabla K(t,\varphi(t)),\psi(t)\rangle\ \ud t = \int_{1}^{+\infty} \langle\nabla U(r_0(t)+\varphi(t)+\Tilde{x}^0)-\nabla U(r_0(t)),\psi(t)\rangle\ \ud t 
\end{displaymath}
for every $\psi\in\mathcal{D}_0^{1,2}(1,+\infty)$. Given a sequence $(\varphi^n)_n\subset\mathcal{D}_0^{1,2}(1,+\infty)$ we have to prove that if $\varphi^n\rightarrow\varphi$ in $\mathcal{D}_0^{1,2}(1,+\infty)$, then 
\begin{displaymath}
    \sup_{\|\psi\|_\mathcal{D}\leq 1} \bigg|\int_{1}^{+\infty} \langle \nabla K(t,\varphi^n(t))-\nabla K(t,\varphi(t)),\psi(t)\rangle\ \ud t\bigg|\rightarrow0.
\end{displaymath}
Since
\begin{displaymath}
    \nabla K(t,\varphi(t)) = \nabla U(r_0(t)+\varphi(t)+\Tilde{x}^0)-\nabla U(r_0(t)) = \int_{0}^{1} \nabla^2 K(t,s\varphi(t))\varphi(t)\ \ud s,
\end{displaymath}
we can estimate
\begin{equation}\label{ex_par_inequality_K}
    \|\nabla K(t,\varphi(t))\|_\mathcal{M} \leq \int_{0}^{1} \|\nabla^2 K(t,s\varphi(t))\|_\mathcal{M} \|\varphi(t)\|_\mathcal{M}\ \ud s \leq C_{10}\frac{\|\varphi(t)\|_\mathcal{M}}{t^2},
\end{equation}
where $C_{10}>0$ is a proper constant. Using the Cauchy-Schwartz inequality we can then compute
\begin{displaymath}
    \begin{split}
    &\sup_{\|\psi\|_\mathcal{D}\leq 1} \bigg|\int_{1}^{+\infty} \langle \nabla K(t,\varphi^n(t))-\nabla K(t,\varphi(t)),\psi(t)\rangle\ \ud t\bigg| \\
    &\leq \sup_{\|\psi\|_\mathcal{D}\leq 1} \int_{1}^{+\infty} t\| \nabla K(t,\varphi^n(t))-\nabla K(t,\varphi(t))\|_\mathcal{M}\frac{\|\psi(t)\|_\mathcal{M}}{t}\ \ud t\\
    &\leq \sup_{\|\psi\|_\mathcal{D}\leq 1} \bigg( \int_{1}^{+\infty} \frac{\|\psi(t)\|_\mathcal{M}^2}{t^2}\ \ud t\bigg)^{1/2} \bigg( \int_{1}^{+\infty} t^2 \| \nabla K(t,\varphi^n(t))-\nabla K(t,\varphi(t))\|_\mathcal{M}^2\ \ud t\bigg)^{1/2} \\
    & \leq 2 \bigg( \int_{1}^{+\infty} t^2 \| \nabla K(t,\varphi^n(t))-\nabla K(t,\varphi(t))\|_\mathcal{M}^2\ \ud t\bigg)^{1/2}.  
    \end{split}
\end{displaymath}
Now, using \eqref{ex_par_inequality_K}
\begin{displaymath}
\begin{split}
    \| \nabla K(t,\varphi^n(t))-\nabla K(t,\varphi(t))\|_\mathcal{M}^2 &= \bigg| \int_{0}^{1} \nabla^2 K(t,\varphi(t)+\sigma(\varphi^n(t)-\varphi))(\varphi^n(t)-\varphi(t))\ \ud \sigma \bigg|^2\\
    & \leq \bigg( \int_{0}^{1} \|\nabla^2 K(t,\varphi(t)+\sigma(\varphi^n(t)-\varphi(t)))(\varphi^n(t)-\varphi(t))\|_\mathcal{M}\ \ud \sigma \bigg)^2\\
    & \leq \bigg( \int_{0}^{1} \frac{\|\varphi^n(t)-\varphi(t)\|_\mathcal{M}}{t^2}\ \ud \sigma \bigg)^2 \\
    & = \frac{\|\varphi^n(t)-\varphi(t)\|_\mathcal{M}^2}{t^4}.
\end{split}
\end{displaymath}
From this last computation, it follows that 
\begin{displaymath}
    \begin{split}
    \bigg( \int_{1}^{+\infty} t^2 \| \nabla K(t,\varphi^n(t))-\nabla K(t,\varphi(t))\|_\mathcal{M}^2\ \ud t\bigg)^{1/2} & \leq \bigg( \int_{1}^{+\infty} \frac{\|\varphi^n(t)-\varphi(t)\|_\mathcal{M}^2}{t^2}\ \ud t\bigg)^{1/2}\\
    & \leq 2 \bigg( \int_{1}^{+\infty} \|\dot{\varphi}^n(t)-\dot{\varphi}(t)\|_\mathcal{M}^2\ \ud t\bigg)^{1/2}\\
    & = 2\|\varphi^n - \varphi\|_\mathcal{D}
    \end{split}
\end{displaymath}
and since $\|\varphi^n - \varphi\|_\mathcal{D}\rightarrow0$ as $n\rightarrow+\infty$, this proves our thesis.

\subsection{Absence of collisions and parabolicity of the motion}

Given a minimizer of the Lagrangian action $\varphi\in\mathcal{D}_0^{1,2}(1,+\infty)$, we apply Marchal's Theorem to state that $\varphi$ has no collisions.

To conclude, we observe that given 
\begin{displaymath}
    x(t) = \varphi(t) + \beta b_m t^{2/3} + \Tilde{x}^0,
\end{displaymath}
we have
\begin{displaymath}
    \dot{x}(t) = \dot{\varphi}(t) + \frac{2}{3}\beta b_m t^{-1/3}.
\end{displaymath}
To prove that the motion $x$ is indeed parabolic, we still have to prove that 
\begin{displaymath}
    \lim_{t\rightarrow+\infty}\dot{x}(t) = \lim_{t\rightarrow+\infty}\dot{\varphi}(t) = 0.
\end{displaymath}

Since $\int_{1}^{+\infty} |\dot{\varphi}_{ij}(t)|^2\ \ud t < +\infty$, we have
\begin{displaymath}
    \liminf_{t\rightarrow+\infty} |\dot{\varphi}_{ij}(t)| = 0.
\end{displaymath}
Because of the conservation of the energy along the motion, we have
\begin{displaymath}
    \frac{1}{2}\|\dot{x}(t)\|^2_\mathcal{M} - U(x(t)) = \frac{1}{2}\sum_{i=1}^{N} m_i \bigg|\dot{\varphi}_i(t) + \frac{2}{3}\beta b_{m_i} t^{-1/3} \bigg|^2 - U(x(t)) = h.
\end{displaymath}
Since there is at least a subsequence $(t_k)_k$, with $t_k\rightarrow+\infty$, such that $\lim_{t_k\rightarrow+\infty} \dot{\varphi}_i(t_k)=0$, it follows that $h=0$ and, consequently, 
\begin{displaymath}
    \frac{1}{2}\|\dot{x}(t)\|^2_\mathcal{M} - U(x(t)) = 0.
\end{displaymath}
From this, we have that $\lim_{t\rightarrow+\infty} \dot{x}(t) = 0$.

\subsection{Asymptotic estimates for half parabolic motions}

In order to give a better description of the asymtptoic expansion of parabolic motions, we can improve inequality \eqref{2.3}. In particular, we can show that for any $\varphi\in\mathcal{D}_0^{1,2}(1,+\infty)$, it holds
\begin{equation}\label{new_estimate_parabolic}
    \|\varphi(t)\|_\mathcal{M} \leq c t^{\frac{1}{3}+\varepsilon},\quad \forall\varepsilon>0,
\end{equation}
for a proper constant $c\in\R$. This section is devoted to the proof of this estimate.

Let us consider a half parabolic motion $x(t)$ having the form \eqref{expr_parabolic_motion}, where $\varphi\in\mathcal{D}_0^{1,2}(1,+\infty)$ is a solution of the equations of motion $\mathcal{M}\ddot{\varphi}(t) = \nabla U(r_0(t) + \varphi(t) + \tilde{x}^0) - \nabla U(r_0(t))$. We can write:
\begin{displaymath}
\begin{split}
    \mathcal{M}\ddot{\varphi}(t) &= \frac{1}{\beta^2 t^{4/3}} \bigg[ \nabla U\bigg(\frac{x(t)}{\beta t^{2/3}}\bigg) - \nabla U\bigg(\frac{r_0(t)}{\beta t^{2/3}}\bigg) \bigg] \\
    & = \frac{1}{\beta^2 t^{4/3}} \bigg[ \nabla U\bigg(b_m + \frac{\varphi(t)}{\beta t^{2/3}} + \frac{\Tilde{x}^0}{\beta t^{2/3}}\bigg) - \nabla U(b_m) \bigg]\\
    & = \frac{1}{\beta^3 t^2}\int_{0}^{1} \nabla^2 U\bigg(b_m + \theta\frac{(\varphi(t)+\tilde{x}^0)}{\beta t^{2/3}}\bigg)(\varphi(t)+\tilde{x}^0)\ \ud\theta\\
    & = \frac{1}{\beta^3 t^2}\bigg[\int_{0}^{1} \nabla^2 U\bigg(b_m + \theta\frac{(\varphi(t)+\tilde{x}^0)}{\beta t^{2/3}}\bigg)\ \ud\theta\bigg](\varphi(t)+\tilde{x}^0),
\end{split}
\end{displaymath}
where we can view the integral term as a matrix.

Fixing a real constant $\delta\in(1,2)$ and a sufficiently big constant $k\in\R$, we define a test function $\psi_k:\R\rightarrow\mathcal{X}$ as
\begin{displaymath}
    \psi_k(t)=\eta^2\min\{k,\|\varphi(t)\|_\mathcal{M}^{\delta-1}\}\varphi(t)
\end{displaymath}
where $\eta:\R\rightarrow\R$ is a $C^\infty$-class cut-off function having the form
\begin{displaymath}
	\eta(t) = \begin{cases}
		0, \quad t\in[1,R] \\
		1, \quad t\in[2R,+\infty)
	\end{cases},
\end{displaymath}
for $R$ big enough, with $0<\eta(t)<1,\ \forall t\in(R,2R)$ . We point out that $k$ can be chosen such that $\eta\equiv1$ when $\|\varphi(t)\|_\mathcal{M}^{\delta-1}>k$, so that we have
\begin{displaymath}
    \dot{\psi}_k(t) =
    \begin{cases}
        2\eta\dot{\eta}\|\varphi(t)\|_\mathcal{M}^{\delta-1}\varphi(t) + \eta^2\delta\|\varphi(t)\|_\mathcal{M}^{\delta-2}\langle\varphi(t),\dot{\varphi}(t)\rangle_\mathcal{M},\qquad & t\in I_k\\
        k\dot{\varphi}(t),\qquad&t\in\hat{I}_k
    \end{cases},
\end{displaymath}
where $I_k = \{t\in[1,+\infty) : \|\varphi(t)\|_\mathcal{M}^{\delta-1}\leq k\}$ and $\hat{I}_k = [1,+\infty)\setminus I_k = \{t\in[1,+\infty) : \|\varphi(t)\|_\mathcal{M}^{\delta-1}> k\}$.

Multiplying the equations of motion for $\psi_k(t)$ and integrating, we obtain
\begin{displaymath}\begin{split}
    &\int_{R}^{+\infty} -\langle\ddot{\varphi}(t),\psi_k(t)\rangle_\mathcal{M} + \bigg\langle\frac{1}{\beta^3 t^2} \bigg[\int_{0}^{1} \nabla^2 U\bigg(b_m + \theta\frac{(\varphi(t)+\tilde{x}^0)}{\beta t^{2/3}}\bigg)\ \ud\theta\bigg](\varphi(t)+\tilde{x}^0),\psi_k\bigg\rangle\ \ud t\\
    & = \int_{R}^{+\infty}\langle\dot{\varphi}(t),\dot{\psi}_k(t)\rangle_\mathcal{M} + \bigg\langle\frac{1}{\beta^3 t^2} \bigg[\int_{0}^{1} \nabla^2 U\bigg(b_m + \theta\frac{(\varphi(t)+\tilde{x}^0)}{\beta t^{2/3}}\bigg)\ \ud\theta\bigg](\varphi(t)+\tilde{x}^0),\psi_k\bigg\rangle\ \ud t.
\end{split}\end{displaymath}

Recalling that $\|\nabla^2 U(r_0 + \theta(\varphi(t)+\tilde{x}^0))\|_\mathcal{M} \leq \frac{C_{11}}{t^2}$ for a proper constant $C_{11}$, for every $t>1$ and for every $\theta\in[0,1]$, we can use Hölder's and Hardy's inequalities to estimate
\begin{displaymath}
    \begin{split}
        & \int_{R}^{+\infty}\langle\dot{\varphi}(t),\dot{\psi}_k(t)\rangle_\mathcal{M} + \bigg\langle \bigg[\int_{0}^{1} \nabla^2 U(r_0(t) + \theta(\varphi(t)+\tilde{x}^0))\ \ud\theta\bigg]\varphi(t),\psi_k(t)\bigg\rangle\ \ud t \\
        & = -\int_{R}^{+\infty}\bigg\langle \bigg[\int_{0}^{1} \nabla^2 U(r_0(t) + \theta(\varphi(t)+\tilde{x}^0))\ \ud\theta\bigg]\tilde{x}^0,\psi_k(t)\bigg\rangle\ \ud t \\
        & \leq C_{11}\int_{R}^{+\infty}\frac{\|\psi_k(t)\|_\mathcal{M}}{t^2}\ \ud t \\
        & \leq C_{11} \int_{R}^{+\infty}\frac{\|\varphi(t)\|_\mathcal{M}^\delta}{t^2}\ \ud t \\
        & = C_{11} \int_{R}^{+\infty}\frac{1}{t^{2-\delta}}\frac{\|\varphi(t)\|_\mathcal{M}^\delta}{t^\delta}\ \ud t \\
        & \leq C_{11} \bigg( \int_{R}^{+\infty}\frac{1}{t^2}\ \ud t\bigg)^{(2-\delta)/2}\bigg( \int_{R}^{+\infty}\frac{\|\varphi(t)\|_\mathcal{M}^2}{t^2}\ \ud t \bigg)^{\delta/2} \\
        & \leq C_{12} \|\varphi\|_\mathcal{D}^\delta,
    \end{split}
\end{displaymath}
where $C_{12}$ is a proper constant.

\begin{rem}
    We recall that the Keplerian potential $U$ is homogeneous of degree -1. For any configuration $x\in\mathcal{X}$, denoting $s=\frac{x}{\|x\|_\mathcal{M}}$,
    \begin{displaymath}
        U(x) = U\bigg(\|x\|_\mathcal{M} \frac{x}{\|x\|_\mathcal{M}}\bigg) = \frac{U(s)}{\|x\|_\mathcal{M}}.
    \end{displaymath}
    The Hessian matrix of $U$ with respect to $x$ is
    \begin{displaymath}
        \nabla^2 U(x) = -\frac{U(s)\mathcal{M}}{\|x\|_\mathcal{M}^3} + 3\frac{U(s)}{\|x\|_\mathcal{M}^5}\mathcal{M}x\otimes \mathcal{M}x - 2\frac{\nabla_s U(s)\otimes \mathcal{M}x}{\|x\|_\mathcal{M}^4} + \frac{\nabla_s^2 U(s)}{\|x\|_\mathcal{M}^3},
    \end{displaymath}
    where $x \otimes x$ denotes the symmetric square matrix with components $(x \otimes x)_{ij} = x_i x_j$ for $i,j \in {1,...,N}$, and $\nabla_s U(s)$ and $\nabla_s^2 U(s)$ represent the gradient and the Hessian matrix of $U$ with respect to $s$, respectively. 
    Choosing $s=b_m$, since $b_m$ is the minimum of the restricted potential, we have $\frac{\nabla_s U(s)\otimes \mathcal{M}x}{\|x\|_\mathcal{M}^4} = 0$. Besides, since $\mathcal{M}x\otimes \mathcal{M}x$ and $\nabla_s^2 U(s)$ are positive semidefinite quadratic forms, it holds
    \begin{equation}\label{diseq_hessian_kepler_potential}
        \nabla^2 U(x) \geq -\frac{U(b_m)\mathcal{M}}{\|x\|_\mathcal{M}^3}.
    \end{equation}
\end{rem}

Using a continuity argument and \eqref{diseq_hessian_kepler_potential}, we can also say that for every $\mu>0$ there is a $\bar{T}>0$ such that, for every $t>\bar{T}$,
\begin{displaymath}
    \frac{1}{\beta^3 t^2} \nabla^2 U\bigg(b_m + \theta\frac{(\varphi(t)+\tilde{x}^0)}{\beta t^{2/3}}\bigg) \geq -\frac{2}{9}(1+\mu)\frac{\mathcal{M}}{t^2}
\end{displaymath}
in the sense of quadratic forms. It follows
\begin{displaymath}
    \begin{split}
        & \int_{R}^{+\infty}\langle\dot{\varphi}(t),\dot{\psi}_k(t)\rangle_\mathcal{M} + \bigg\langle\frac{1}{\beta^3 t^2} \bigg[\int_{0}^{1} \nabla^2 U\bigg(b_m + \theta\frac{(\varphi(t)+\tilde{x}^0)}{\beta t^{2/3}}\bigg)\ \ud\theta\bigg]\varphi(t),\psi_k(t)\bigg\rangle\ \ud t \\
        & \geq \int_{R}^{+\infty}\langle\dot{\varphi}(t),\dot{\psi}_k(t)\rangle_\mathcal{M} -\frac{2}{9}(1+\mu)\bigg\langle\frac{\varphi(t)}{t^2},\psi_k(t)\bigg\rangle_\mathcal{M}\ \ud t.
    \end{split}
\end{displaymath}
To estimate the right-hand side of the last inequality, we study the integral separately on the two complementary sets $I_k$ and $\hat{I}_k$. In $I_k$, we have
\begin{displaymath}
    \begin{split}
        & \int_{I_k}\langle\dot{\varphi}(t),\dot{\psi}_k(t)\rangle_\mathcal{M} -\frac{2}{9}(1+\mu)\bigg\langle\frac{\varphi(t)}{t^2},\psi_k(t)\bigg\rangle_\mathcal{M}\ \ud t \\
        & = \int_{I_k} 2\eta\dot{\eta}\|\varphi(t)\|_\mathcal{M}^{\delta-1}\langle\dot{\varphi}(t),\varphi(t)\rangle_\mathcal{M} + \eta^2\delta\|\varphi(t)\|_\mathcal{M}^{\delta-1}\|\dot{\varphi(t)}\|_\mathcal{M} -\frac{2}{9}(1+\mu)\eta^2\frac{\|\varphi(t)\|_\mathcal{M}^{\delta+1}}{t^2}\ \ud t,
    \end{split}
\end{displaymath}
which implies
\begin{displaymath}
    \int_{I_k} \eta^2\delta\|\varphi(t)\|_\mathcal{M}^{\delta-1}\|\dot{\varphi}(t)\|_\mathcal{M} -\frac{2}{9}(1+\mu)\eta^2\frac{\|\varphi(t)\|_\mathcal{M}^{\delta+1}}{t^2}\ \ud t \leq \int_{I_k} 2\eta\dot{\eta}\|\varphi(t)\|_\mathcal{M}^{\delta}\|\dot{\varphi}(t)\|_\mathcal{M}\ \ud t + C_{12}\|\varphi\|_\mathcal{D}^\delta. 
\end{displaymath}
where the cut-off function makes sure that the last integral is finite. Besides, we also have
\begin{displaymath}
    \begin{split}
        & \int_{I_k} \eta^2\delta\|\varphi(t)\|_\mathcal{M}^{\delta-1}\|\dot{\varphi}(t)\|_\mathcal{M} -\frac{2}{9}(1+\mu)\eta^2\frac{\|\varphi(t)\|_\mathcal{M}^{\delta+1}}{t^2}\ \ud t \\
        & = \int_{I_k} \frac{4\delta}{(\delta+1)^2}\bigg(\eta\frac{\ud}{\ud t}\|\varphi(t)\|_\mathcal{M}^{\frac{\delta+1}{2}}\bigg)^2 -\frac{2}{9}(1+\mu)\eta^2\frac{\|\varphi(t)\|_\mathcal{M}^{\delta+1}}{t^2}\ \ud t.
    \end{split}
\end{displaymath}

On the other hand, working on the interval $\hat{I}_k$ we obtain
\begin{displaymath}
    \begin{split}
        & \int_{\hat{I}_k}\langle\dot{\varphi}(t),\dot{\psi}_k(t)\rangle_\mathcal{M} -\frac{2}{9}(1+\mu)\bigg\langle\frac{\varphi(t)}{t^2},\psi_k(t)\bigg\rangle_\mathcal{M}\ \ud t \\
        & = \int_{\hat{I}_k} k\|\dot{\varphi}(t)\|_\mathcal{M}^2 - \frac{2}{9}(1+\mu)k\frac{\|\varphi(t)\|_\mathcal{M}^2}{t^2}\ \ud t \\
        & \geq \int_{\hat{I}_k} \frac{4\delta}{(\delta+1)^2} k\|\dot{\varphi}(t)\|_\mathcal{M}^2 - \frac{2}{9}(1+\mu)k\frac{\|\varphi(t)\|_\mathcal{M}^2}{t^2}\ \ud t,
    \end{split}
\end{displaymath}
where we used the fact that $\frac{4\delta}{(\delta+1)^2} < 1$ for every $\delta\in(1,2)$.

Now, we define a function $u_k:\R\rightarrow\R$ as
\begin{displaymath}
    u_k(t) = \min\{\eta\|\varphi(t)\|_\mathcal{M}^{\frac{\delta-1}{2}},k^{1/2}\}\|\varphi(t)\|_\mathcal{M}.
\end{displaymath}

Putting everything together, we can use Hardy's inequality to say that 
\begin{displaymath}
    \begin{split}
        &\int_{I_k} \frac{4\delta}{(\delta+1)^2}\bigg(\eta\frac{\ud}{\ud t}\|\varphi(t)\|_\mathcal{M}^{\frac{\delta+1}{2}}\bigg)^2 -\frac{2}{9}(1+\mu)\eta^2\frac{\|\varphi(t)\|_\mathcal{M}^{\delta+1}}{t^2}\ \ud t + \int_{\hat{I}_k} \frac{4\delta}{(\delta+1)^2} k\|\dot{\varphi}(t)\|_\mathcal{M}^2 - \frac{2}{9}(1+\mu)k\frac{\|\varphi(t)\|_\mathcal{M}^2}{t^2}\ \ud t \\
        & = \int_{1}^{+\infty} \frac{4\delta}{(\delta+1)^2} \|\dot{u}_k(t)\|_\mathcal{M}^2 - \frac{2}{9}(1+\mu)\frac{\|u_k(t)\|_\mathcal{M}^2}{t^2}\ \ud t \\
        & \geq \int_{1}^{+\infty} \bigg(\frac{4\delta}{(\delta+1)^2} - \frac{8}{9}(1+\mu)\bigg)\|\dot{u}_k(t)\|_\mathcal{M}^2\ \ud t.
    \end{split}
\end{displaymath}
In particular, we can choose $\mu$ such that $\frac{4\delta}{(\delta+1)^2} - \frac{8}{9}(1+\mu) > 0$, which proves that $u_k\in\mathcal{D}_0^{1,2}(1,+\infty)$. 

Since the estimates we obtained do not depend on $k$, we can take $k\rightarrow+\infty$ so that \eqref{2.3} leads us to the conclusion of our proof. We have thus shown that for any $\varphi\in\mathcal{D}_0^{1,2}(1,+\infty)$ and for any $\delta\in(1,2)$ there is a constant $c$, which depends on $\delta$ and $\|\varphi\|_\mathcal{D}$, such that 
\begin{displaymath}
    \|\varphi(t)\|_\mathcal{M} \leq c t^{\frac{1}{\delta+1}},\qquad\forall t\geq1.
\end{displaymath}

\section{Existence of minimal half hyperbolic-parabolic motions}\label{sec_hyperbolic_parabolic}

This last section is devoted to the proof of Theorem \ref{thm_partially_hyperbolic}. To prove the existence of hyperbolic-parabolic solutions in the $N$-body problem, we will use the cluster decomposition that we briefly introduced in Section \ref{sec_intro} to decompose the Lagrangian action, so that we will finally be able to minimize the renormalized action over the set $\mathcal{D}_0^{1,2}(1,+\infty)$. 
\begin{defn}\label{def_cluster_partition}
    Given a configuration $a\in\mathcal{X}$ and a motion $x(t)=at+O(t^{2/3})$ as $t\rightarrow+\infty$, its corresponding natural partition ($a$-partition) of the index set $\mathcal{N}=\{1,...,N\}$ is the one for which $i,j\in\mathcal{N}$ belong to the same class if and only if the mutual distance $|r_i(t)-r_j(t)|$ grows as $O(t^{2/3})$. Equivalently, if $a = (a_1,...,a_N)$, then the natural partition is defined by the relation $i \sim j$ if and only if $a_i = a_j$. The partition classes will be called clusters.
\end{defn}

We give now some definitions and basic notations related to a given partition  $\mathcal{P}$ of the set $\mathcal{N}=\{1,...,N\}$.
\begin{defn}
    Let $\mathcal{P}$ be a given partition of $\mathcal{N}$ and consider a configuration $x=(r_1,...,r_N)\in \mathcal{X}$. For each cluster $K\in\mathcal{P}$ we define the mass of the cluster as
    \begin{displaymath}
        M_K=\sum_{i\in K}m_i.
    \end{displaymath}
    Besides, for any couple of clusters $K_1,K_2\in\mathcal{P}$, $K_1\neq K_2$, we define the mass of the two clusters as
    \begin{displaymath}
        M_{K_{1,2}}=\sum_{i\in K_1 \cup K_2}m_i.
    \end{displaymath}
\end{defn}
\begin{defn}
    Let $\mathcal{P}$ be any given partition of $\mathcal{N}$. Then, for every given curve $x(t)=(r_1(t),...,r_N(t))$ in $\mathcal{X}$ and for each cluster $K\in\mathcal{P}$ we define the function
    \begin{displaymath}
        U_K(t) = \sum_{i,j\in K,\ i<j}\frac{m_i m_j}{|r_i(t)-r_j(t)|},
    \end{displaymath}
    which represents the restriction of the potential $U$ to the cluster $K$.
\end{defn}

The system we are studying here is
\begin{displaymath}
\begin{cases}
    \mathcal{M}\ddot{x}=\nabla U(x)\\
    x(1)=x^0\\
    \lim_{t\rightarrow+\infty}\dot{x}(t)=a
\end{cases},
\end{displaymath}
where $x^0\in\mathcal{X}$ and $a$ is a configuration with collisions. We will look for solutions of the form $x(t) = \varphi(t)+\gamma_0(t)$, for any $\varphi\in\mathcal{D}_0^{1,2}(1,+\infty)$, where $\gamma_0$ is a proper function, so that our problem equivalently reads
\begin{displaymath}
\begin{cases}
    \mathcal{M}\ddot{\varphi}=\nabla U(\varphi+\gamma_0)-\ddot{\gamma}_0\\
    \varphi(0)=0\\
    \lim_{t\rightarrow+\infty}\dot{\varphi}(t)=0
\end{cases}.
\end{displaymath}
We can thus prove the existence of solutions to the last system by minimizing the associated renormalized Lagrangian action.

We partition the indexes according to the natural cluster partition, so that we obtain a partition of $\mathcal{N}$ of the form
\begin{displaymath}
    K_1:=\{1,...,k_1\},\ K_2:=\{k_1+1,...,k_2\},\ K_3:=\{k_2+1,...,k_3\}... 
\end{displaymath}
For every $K_i$, we can choose a central configuration $b^{K_i}$ which is minimal for that particular cluster and we can define the configuration
\begin{displaymath}
    b=(b^{K_1},b^{K_2},...)\in \mathcal{X}.
\end{displaymath}
Using this particular definition of $b$, we can then look for solutions of the form
\begin{equation}\label{explicit_hyperbolic-parabolic}
    x(t) = \varphi(t) + at +\beta b t^{2/3} -a -b +x^0 = \varphi(t) + at +\beta b t^{2/3} + \Tilde{x}^0.
\end{equation}
Here, $\beta$ is a real vector with as many components as the number of clusters. Precisely, we have
\begin{displaymath}
    \beta=(\beta_{K_1},\beta_{K_2},...),
\end{displaymath}
with 
\begin{displaymath}
    \beta_{K_1}=\sqrt[3]{\frac{9}{2}U_{min}^{K_1}}
\end{displaymath}and $U_{min}^{K_1}$ denotes the minimum of the potential $U$ restricted to the first cluster; 
\begin{displaymath}
    \beta_{K_2}=\sqrt[3]{\frac{9}{2}U_{min}^{K_2}}
\end{displaymath}
and $U_{min}^{K_2}$ denotes the minimum of the potential $U$ restricted to the second cluster, and so on. With an abuse of notation, in this section we write $\beta b$ to denote the configuration $(\beta_{K_1}b^{K_1}, \beta_{K_2}b^{K_2},...)\in\mathcal{X}$.

Using the aforementioned partition of the bodies, it is possible to decompose the Lagrangian action of a curve: for every $\varphi\in\mathcal{D}_0^{1,2}(1,+\infty)$, we define
\begin{equation}\label{action_partially_hyperbolic}
\begin{split}
    \mathcal{A}(\varphi) & := \sum_{K\in\mathcal{P}} \mathcal{A}_{K}(\varphi) + \sum_{K_1,K_2\in\mathcal{P},\ K_1\neq K_2} \mathcal{A}_{K_1,K_2}(\varphi)\\
    &  = \sum_{K\in\mathcal{P}}\bigg(\sum_{i,j\in K,\ i<j} \mathcal{A}_{K}^{ij}(\varphi)\bigg) + \frac{1}{2}\sum_{K_1,K_2\in\mathcal{P},\ K_1\neq K_2}\bigg(\sum_{i\in K_1,\ j\in K_2} \mathcal{A}_{K_1,K_2}^{ij}(\varphi)\bigg),
\end{split}
\end{equation}
where
\begin{equation}\label{action_partially_hyperbolic_inside}
\begin{split}
    \mathcal{A}_{K}^{ij}(\varphi) := & \int_{1}^{+\infty} \frac{1}{2M_K} m_i m_j|\dot{\varphi}_i(t)-\dot{\varphi}_j(t)|^2 + \frac{m_i m_j}{|\varphi_{ij}(t)+a_{ij}t+\beta_K b^K_{ij}t^{2/3}+\Tilde{x}^0_{ij}|} - \frac{m_i m_j}{|\beta_K b^K_{ij}t^{2/3}|} \\
    &+ \frac{2}{9}\frac{\beta_K}{M_K}m_i m_j \frac{\langle b^K_{ij},\varphi_{ij}(t)\rangle}{t^{4/3}}\ \ud t,
\end{split}
\end{equation}
\begin{equation}\label{action_partially_hyperbolic_outside}
\begin{split}
    \mathcal{A}_{K_1,K_2}^{ij}(\varphi) := &\int_{1}^{+\infty}\frac{1}{2M_{K_{1,2}}} m_i m_j|\dot{\varphi}_i(t)-\dot{\varphi}_j(t)|^2+ \frac{m_i m_j}{|\varphi_{ij}(t)+a_{ij}t+\beta_{K_{1,2}} b^{K_{1,2}}_{ij}t^{2/3}+\Tilde{x}^0_{ij}|} - \frac{m_i m_j}{|a_{ij}t|}\ \ud t.
\end{split}
\end{equation}
Here, we used the notations:
\begin{displaymath}
    \begin{split}
        &b^{K_{1,2}}=(b^{K_1},b^{K_2})\\
        &\beta_{K_{1,2}}b^{K_{1,2}}=(\beta_{K_1}b^{K_1},\beta_{K_2}b^{K_2})
    \end{split}
\end{displaymath}

We point out that the term \eqref{action_partially_hyperbolic_inside} is the part of the Lagrangian action that refer to the (parabolic) motion of the bodies inside each cluster, while the term \eqref{action_partially_hyperbolic_outside} refers to the (linear) motion of the cluster. In the following sections, we will study the two terms separately, in order to apply the Direct Method.

\subsection{Coercivity of $\mathcal{A}(\varphi)$}

We start with the proof of the coercivity of the Lagrangian action when restricted to a general cluster, where we denote by $K$ the set of indexes related to this cluster. Because of the natural cluster partition of the bodies, we have $a_i = a_j$ for any $i,j\in K$. This means that for any $\varphi\in\mathcal{D}_0^{1,2}(1,+\infty)$,
 \begin{displaymath}
 \begin{split}
     \mathcal{A}_{K}(\varphi) = \sum_{i,j\in K,\ i<j} &\int_{1}^{+\infty} \frac{1}{2M_K} m_i m_j|\dot{\varphi}_i(t)-\dot{\varphi}_j(t)|^2 + \frac{m_i m_j}{|\varphi_{ij}(t)+\beta_K b^K_{ij}t^{2/3}+\Tilde{x}^0_{ij}|} - \frac{m_i m_j}{|\beta_K b^K_{ij}t^{2/3}|} \\
     &+ \frac{2}{9}\frac{\beta_K}{M_K}m_i m_j \frac{\langle b^K_{ij},\varphi_{ij}(t)\rangle}{t^{4/3}}\ \ud t.
\end{split} 
\end{displaymath}
Using the homogeneity of the potential and denoting by $U_K$ the potential $U$ when restricted to the cluster $K$, we apply the inequality
\begin{displaymath}
    U_K(x) \geq \frac{U_K(b^K)}{\|x\|_\mathcal{M}} = \frac{U_{min}}{\|x\|_\mathcal{M}}
\end{displaymath}
to every configuration $x$ restricted to the cluster $K$. It follows
\begin{displaymath}
    \begin{split}
        \mathcal{A}_K(\varphi) \geq &\int_{1}^{+\infty} \sum_{i,j\in K,\ i<j} \bigg(\frac{1}{2M_K} m_i m_j|\dot{\varphi}_i(t)-\dot{\varphi}_j(t)|^2\bigg) + \frac{U_{min}}{\|\varphi(t) + \beta_K b^K t^{2/3}+\Tilde{x}^0\|_\mathcal{M}} - \frac{U_{min}}{\|\beta_K b^K t^{2/3}\|_\mathcal{M}}\\
        &+ \frac{2}{9}\frac{\beta_K}{M_K}\langle \mathcal{M}_K b^K,\varphi(t)\rangle\ \ud t,
    \end{split}
\end{displaymath}
where $\mathcal{M}_K$ denotes the matrix of the masses of the cluster $K$. Using the inequality
\begin{displaymath}
\begin{split}
    \frac{1}{\|\varphi(t) + \beta_K b^K t^{2/3}+\Tilde{x}^0\|_\mathcal{M}}  \geq &\frac{1}{\|\beta_K b^K t^{2/3}\|_\mathcal{M}} - \frac{1}{2\|\beta_K b^K\|_\mathcal{M}^3 t^2}(2t^{2/3}\beta_K\langle \mathcal{M}_K b^K,\varphi(t)\rangle \\ & + 2\langle \mathcal{M}_K\varphi(t),\Tilde{x}^0\rangle + 2t^{2/3}\beta_K\langle\mathcal{M}_K b^K,\Tilde{x}^0\rangle  + \|\varphi(t)\|_\mathcal{M}^2+\|\Tilde{x}^0\|_\mathcal{M}^2),
\end{split}
\end{displaymath}
which holds because of the convexity of the norm, we obtain
\begin{displaymath}
    \begin{split}
        \mathcal{A}_K(\varphi) &\geq \int_{1}^{+\infty} \sum_{i,j\in K,\ i<j} \frac{1}{2M_K} m_i m_j|\dot{\varphi}_{ij}(t)|^2 + \frac{2}{9}\frac{\beta_K}{M_K}\langle \mathcal{M}_K b^K,\varphi(t)\rangle \\
        & \hspace{4mm}- \frac{U_{min}}{2\beta_K^3\|b^K\|_\mathcal{M}^3 t^2}(2t^{2/3}\beta_K\langle \mathcal{M}_K b^K,\varphi(t)\rangle + 2\langle \mathcal{M}_K\varphi(t),\Tilde{x}^0\rangle + 2t^{2/3}\beta_K\langle\mathcal{M}_K b^K,\Tilde{x}^0\rangle + \|\varphi(t)\|_\mathcal{M}^2+\|\Tilde{x}^0\|_\mathcal{M}^2)\ \ud t\\
        & = \int_{1}^{+\infty} \frac{1}{2}\|\dot{\varphi}(t)\|_\mathcal{M}^2 \\
        & \hspace{4mm} - \frac{U_{min}}{2\beta_K^3\|b^K\|_\mathcal{M}^3 t^2}(2\langle \mathcal{M}_K\varphi(t),\Tilde{x}^0\rangle + 2t^{2/3}\beta_K\langle\mathcal{M}_K b^K,\Tilde{x}^0\rangle + \|\varphi(t)\|_\mathcal{M}^2+\|\Tilde{x}^0\|_\mathcal{M}^2)\ \ud t
    \end{split}
\end{displaymath}
We notice that the term
\begin{displaymath}
    C_{13}:=\int_{1}^{+\infty} - \frac{U_{min}}{2\beta_K^3\|b^K\|_\mathcal{M}^3 t^2}( 2t^{2/3}\beta_K\langle\mathcal{M}_K b^K,\Tilde{x}^0\rangle + \|\Tilde{x}^0\|_\mathcal{M}^2)\ \ud t
\end{displaymath}
is constant and finite. Using Hardy and Cauchy-Schwartz inequalities we also have
\begin{displaymath}
    \begin{split}
        -\int_{1}^{+\infty} \frac{U_{min}}{\beta_K^3\|b^K\|_\mathcal{M}^3 t^2} \langle \mathcal{M}_K\varphi(t),\Tilde{x}^0\rangle\ \ud t &= -\frac{2}{9} \int_{1}^{+\infty} \frac{1}{t^2} \langle \mathcal{M}_K\varphi(t),\Tilde{x}^0\rangle\ \ud t\\
        & \geq - \frac{2}{9} \bigg(\int_{1}^{+\infty} \frac{\|\varphi(t)\|_\mathcal{M}^2}{t^2}\ \ud t\bigg)^{1/2}\bigg( \int_{1}^{+\infty} \frac{\|\Tilde{x}^0\|_\mathcal{M}^2}{t^2}\ \ud t\bigg)^{1/2} \\
        & \geq -C_{14} \|\varphi\|_\mathcal{D},
    \end{split}
\end{displaymath}
where $C_{14}:= \frac{8}{9}\big(\int_{1}^{+\infty} \frac{\|\Tilde{x}^0\|_\mathcal{M}^2}{t^2}\ \ud t\big)^{1/2} < +\infty$. Again by Hardy inequality, we obtain
\begin{displaymath}
        \mathcal{A}_K(\varphi) \geq \frac{1}{18}\|\varphi\|_\mathcal{D}^2- C_{14}\|\varphi\|_\mathcal{D}+C_{13},
\end{displaymath}
which implies that the functional $\mathcal{A}_K$ is coercive.\\

We now focus on studying the terms
\begin{displaymath}
\begin{split}
    \mathcal{A}_{K_1,K_2}(\varphi) := \sum_{i\in K_1,\ j\in K_2} &\int_{1}^{+\infty}\frac{1}{2M_{K_{1,2}}} m_i m_j|\dot{\varphi}_i(t)-\dot{\varphi}_j(t)|^2+ \frac{m_i m_j}{|\varphi_{ij}(t)+a_{ij}t+\beta_{K_{1,2}} b^{K_{1,2}}_{ij}t^{2/3}+\Tilde{x}^0_{ij}|} - \frac{m_i m_j}{|a_{ij}t|}\ \ud t.
\end{split} 
\end{displaymath}

\begin{rem}
    We notice that if two bodies of the configuration $b^{K_{1,2}}$  belong to different clusters and have collisions, that is, if there are $i\in K_1$ and $j\in K_2$ such that $b^{K_{1,2}}_i=b^{K_{1,2}}_j$, then the functional reads
    \begin{displaymath}
        \mathcal{A}_{K_1,K_2}(\varphi) = \sum_{i\in K_1,\ j\in K_2} \int_{1}^{+\infty} \frac{1}{2M_{K_{1,2}}} m_i m_j|\dot{\varphi}_i(t)-\dot{\varphi}_j(t)|^2+\frac{m_i m_j}{|\varphi_{ij}(t)+a_{ij}t+\Tilde{x}^0_{ij}|} - \frac{m_i m_j}{|a_{ij}t|}\ \ud t.
    \end{displaymath}
    Since $a_i\neq a_j$ when $i\in K_1$, $j\in K_2$ and $K_1\neq K_2$, we have already proved that in this case the action functional $\mathcal{A}$ is coercive.
\end{rem}

Assuming $b^{K_{1,2}}$ without collisions, we proceed in the following way. By the triangular inequality, we have
\begin{displaymath}\begin{split}
    & \int_{1}^{+\infty} \frac{1}{|\varphi_{ij}(t)+a_{ij}t+\beta_{K_{1,2}}b^{K_{1,2}}_{ij}t^{2/3}+\Tilde{x}^0_{ij}|} - \frac{1}{|a_{ij}t|}\ \ud t \\
    &\geq \int_{1}^{+\infty} \frac{1}{\|\varphi_{ij}\|_\mathcal{D} t^{1/2}+|a_{ij}|t+\beta_{K_{1,2}}|b^{K_{1,2}}_{ij}|t^{2/3} + |\Tilde{x}^0_{ij}|} - \frac{1}{|a_{ij}|t}\ \ud t.
\end{split}    
\end{displaymath}
Using the changes of variables $s=\|\varphi\|_\mathcal{D}u$, we obtain
\begin{displaymath}\begin{split}
    &\int_{1}^{+\infty} \frac{1}{\|\varphi_{ij}\|_\mathcal{D} t^{1/2}+|a_{ij}|t+\beta_{K_{1,2}}|b^{K_{1,2}}_{ij}|t^{2/3} + |\Tilde{x}^0_{ij}|} - \frac{1}{|a_{ij}|t}\ \ud t\\
    & = 2 \int_{1}^{+\infty} \bigg(\frac{1}{\|\varphi_{ij}\|_\mathcal{D} s+|a_{ij}|s^2+\beta_{K_{1,2}}|b^{K_{1,2}}_{ij}|s^{4/3} + |\Tilde{x}^0_{ij}|} - \frac{1}{|a_{ij}|s^2}\bigg)s\ \ud s\\
    & = \frac{2}{\|\varphi\|_\mathcal{D}|a_{ij}|} \int_{1}^{+\infty} \bigg(\frac{1}{s^2+\frac{\beta_{K_{1,2}}|b^{K_{1,2}}_{ij}|}{|a_{ij}|}\frac{s^{4/3}}{\|\varphi\|_\mathcal{D}^{2/3}}+\frac{s}{|a_{ij}|\|\varphi\|_\mathcal{D}}+\frac{|\Tilde{x}^0_{ij}|}{|a_{ij}|\|\varphi\|_\mathcal{D}}} -\frac{1}{\frac{s^2}{\|\varphi\|_\mathcal{D}^2}}\bigg)s\ \ud s\\
    & = \frac{2}{|a_{ij}|} \int_{1/\|\varphi\|_\mathcal{D}}^{+\infty}\bigg( \frac{1}{u^2+\frac{\beta_{K_{1,2}}|b^{K_{1,2}}_{ij}|}{|a_{ij}|}\frac{u^{4/3}}{\|\varphi\|_\mathcal{D}^{2/3}}+\frac{u}{|a_{ij}|}+\frac{|\Tilde{x}^0_{ij}|}{|a_{ij}|\|\varphi\|_\mathcal{D}}} -\frac{1}{u^2}\bigg)u\ \ud u.
\end{split}    
\end{displaymath}
We can observe that for $\|\varphi\|_\mathcal{D}\rightarrow+\infty$, we have $\frac{\beta_{K_{1,2}}|b^{K_{1,2}}_{ij}|}{|a_{ij}|\|\varphi\|_\mathcal{D}^{2/3}}\leq 1$ and $\frac{|\Tilde{x}^0_{ij}|}{|a_{ij}|\|\varphi\|_\mathcal{D}} \leq 1$. So, for $\|\varphi\|_\mathcal{D}\rightarrow+\infty$, it follows
\begin{displaymath}\begin{split}
    & \frac{2}{|a_{ij}|} \int_{1/\|\varphi\|_\mathcal{D}}^{+\infty}\bigg( \frac{1}{u^2+\frac{\beta_{K_{1,2}}|b^{K_{1,2}}_{ij}|}{|a_{ij}|}\frac{u^{4/3}}{\|\varphi\|_\mathcal{D}^{2/3}}+\frac{u}{|a_{ij}|}+\frac{|\Tilde{x}^0_{ij}|}{|a_{ij}|\|\varphi\|_\mathcal{D}}} -\frac{1}{u^2}\bigg)u\ \ud u \\
    & \geq \frac{2}{|a_{ij}|} \int_{1/\|\varphi\|_\mathcal{D}}^{+\infty}\bigg( \frac{1}{u^2+u^{4/3}+\frac{u}{|a_{ij}|}+1} -\frac{1}{u^2}\bigg)u\ \ud u\\
    & = \frac{2}{|a_{ij}|} \int_{1/\|\varphi\|_\mathcal{D}}^{+\infty} \frac{1}{u} \bigg( \frac{1}{1+u^{-2/3}+\frac{u^{-1}}{|a_{ij}|}+u^{-1}} -1 \bigg)\ \ud u.
\end{split}    
\end{displaymath}
Since $1/\|\varphi\|_\mathcal{D}\leq1$ when $\|\varphi\|_\mathcal{D}\rightarrow+\infty$, we can study the integral separately on the intervals $[1/\|\varphi\|_\mathcal{D},1]$ and $[1,+\infty)$. On the second interval, the integral is constant (let us say that it is equal to a constant $C_{15})$. On the other interval, we have
\begin{displaymath}\begin{split}
    & \frac{2}{|a_{ij}|} \int_{1/\|\varphi\|_\mathcal{D}}^{1} \frac{1}{u} \bigg( \frac{1}{1+u^{-2/3}+\frac{u^{-1}}{|a_{ij}|}+u^{-1}} -1 \bigg)\ \ud u \geq \frac{2}{|a_{ij}|} \int_{1/\|\varphi\|_\mathcal{D}}^{1} -\frac{\ud u}{u}. 
\end{split}    
\end{displaymath}
We have thus demonstrated that 
\begin{displaymath}
    \begin{split}
        \int_{1}^{+\infty} \frac{1}{|\varphi_{ij}(t)+a_{ij}t+\beta_{K_{1,2}}b^{K_{1,2}}_{ij}t^{2/3}+\Tilde{x}^0_{ij}|} - \frac{1}{|a_{ij}t|}\ \ud t \geq \frac{2}{|a_{ij}|} \log{\frac{1}{\|\varphi\|_\mathcal{D}}}+C_{15} = -\frac{2}{|a_{ij}|}\log{\|\varphi\|_\mathcal{D}}+C_{15},
    \end{split}
\end{displaymath}
which concludes the proof of the coercivity of the Lagrangian action.

\subsection{Weak lower semicontinuity of $\mathcal{A}(\varphi)$}

In order to prove the weak lower semicontinuity of the Lagrangian action, we can use the decomposition \eqref{action_partially_hyperbolic} and study the weak lower semicontinuity of the terms $\mathcal{A}_K$ and $\mathcal{A}_{K_1,K_2}$ separately, given arbitrary clusters $K,K_1,K_2\in\mathcal{P}$.

Concerning the term $\mathcal{A}_K$, we can refer to Section \ref{sec_parabolic}, since our choice of $\beta_K b^K$ leads us to the same computations.

For the proof of the weak lower semicontinuity of the terms $\mathcal{A}_{K_1,K_2}$, let us consider a sequence $(\varphi^n)_n\subset\mathcal{D}_0^{1,2}(1,+\infty)$ converging weakly in $\mathcal{D}_0^{1,2}(1,+\infty)$ to some $\varphi$, as $n\rightarrow+\infty$. It follows that there is a constant $k\in\R$ such that $\|\varphi^n\|_\mathcal{D}\leq k$ and $\|\varphi\|_\mathcal{D}\leq k$ for every $n\in\N$. We would like to use the inequality
\begin{equation}\label{alternative_proof_second_equality_with_s}
    \frac{1}{|\varphi^n_{ij}(t) + a_{ij}t + \beta_{K_{1,2}}b^{K_{1,2}}_{ij}t^{2/3} + \Tilde{x}^0_{ij}|} - \frac{1}{|a_{ij}t|} = \int_{0}^{1} \frac{\ud}{\ud s}\bigg[\frac{1}{|a_{ij}t + s(\varphi^n_{ij}(t) + \beta_{K_{1,2}}b^{K_{1,2}}_{ij}t^{2/3} + \Tilde{x}^0_{ij})|}\bigg]\ \ud s,
\end{equation}
which holds true when the denominator of the integrand is not zero. For all $s\in(0,1)$ we have
\begin{displaymath}\begin{split}
     |a_{ij}t + s(\varphi^n_{ij}(t) + \beta_{K_{1,2}}b^{K_{1,2}}_{ij}t^{2/3} + \Tilde{x}^0_{ij})| & \geq |a_{ij}|t - s(\|\varphi^n_{ij}\|_\mathcal{D}t^{1/2} + |\beta_{K_{1,2}}b^{K_{1,2}}_{ij}|t^{2/3} + |\Tilde{x}^0_{ij}|) \\
     & > |a_{ij}|t - (\|\varphi^n_{ij}\|_\mathcal{D}t^{1/2} + |\beta_{K_{1,2}}b^{K_{1,2}}_{ij}|t^{2/3} + |\Tilde{x}^0_{ij}|),
\end{split}\end{displaymath}
and since $\|\varphi^n_{ij}\|_\mathcal{D}\leq k$, we have
\begin{displaymath}
    |a_{ij}t + s(\varphi^n_{ij}(t) + \beta_{K_{1,2}}b^{K_{1,2}}_{ij}t^{2/3} + \Tilde{x}^0_{ij})| > |a_{ij}|t - (k t^{1/2} + |\beta_{K_{1,2}}b^{K_{1,2}}_{ij}|t^{2/3} + |\Tilde{x}^0_{ij}|),
\end{displaymath}
where the last term is larger then zero if $t\geq\bar{T}=\bar{T}(k)$, for a proper $\bar{T}$. We can thus study the weak lower semicontinuity of the potential term separately on the two intervals $[1,\bar{T}]$ and $[\bar{T},+\infty)$.

On $[1,\bar{T}]$, the weak lower semicontinuity easily follows from Fatou's Lemma. On $[\bar{T},+\infty)$, we can use \eqref{alternative_proof_second_equality_with_s}:
\begin{displaymath}\begin{split}
    &\int_{\bar{T}}^{+\infty} \frac{1}{|\varphi^n_{ij}(t) + a_{ij}t + \beta_{K_{1,2}}b^{K_{1,2}}_{ij}t^{2/3} + \Tilde{x}^0_{ij}|} - \frac{1}{|a_{ij}t|}\ \ud t \\
    & = \int_{\bar{T}}^{+\infty} \bigg( \int_{0}^{1} -\frac{[a_{ij}t + s(\varphi^n_{ij}(t) + \beta_{K_{1,2}}b^{K_{1,2}}_{ij}t^{2/3} + \Tilde{x}^0_{ij})](\varphi^n_{ij}(t) + \beta_{K_{1,2}}b^{K_{1,2}}_{ij}t^{2/3} + \Tilde{x}^0_{ij})}{|a_{ij}t + s(\varphi^n_{ij}(t) + \beta_{K_{1,2}}b^{K_{1,2}}_{ij}t^{2/3} + \Tilde{x}^0_{ij})|^3}\ \ud s \bigg)\ \ud t.
\end{split}\end{displaymath}
Using \eqref{alternative_proof_inequality}, we then have
\begin{displaymath}
    \begin{split}
    &\int_{\bar{T}}^{+\infty}\bigg| \frac{1}{|\varphi^n_{ij}(t) + a_{ij}t + \beta_{K_{1,2}}b^{K_{1,2}}_{ij}t^{2/3} + \Tilde{x}^0_{ij}|} - \frac{1}{|a_{ij}t|}\bigg|\ \ud t \\
    & \leq \int_{\bar{T}}^{+\infty} \bigg( \int_{0}^{1} \frac{|\varphi^n_{ij}(t) + \beta_{K_{1,2}}b^{K_{1,2}}_{ij}t^{2/3} + \Tilde{x}^0_{ij}|}{|a_{ij}t + s(\varphi^n_{ij}(t) + \beta_{K_{1,2}}b^{K_{1,2}}_{ij}t^{2/3} + \Tilde{x}^0_{ij})|^2}\ \ud s\bigg)\ \ud t\\
    & \leq \int_{\bar{T}}^{+\infty} \bigg( \int_{0}^{1} \frac{3(|kt^{1/2} + \beta_{K_{1,2}}b^{K_{1,2}}_{ij}t^{2/3}| + |\Tilde{x}^0_{ij}|)}{|a_{ij}t|^2 - s|kt^{1/2} + \beta_{K_{1,2}}b^{K_{1,2}}_{ij}t^{2/3} + \Tilde{x}^0_{ij}|^2}\ \ud s\bigg)\ \ud t\\
    & \leq \int_{\bar{T}}^{+\infty} \bigg( \int_{0}^{1} \frac{3k't^{2/3}}{|a_{ij}|^2 t^2 - sk't^{4/3}}\ \ud s\bigg)\ \ud t,
    \end{split}
\end{displaymath}
where $k'\in\R$ is big enough so that $|kt^{1/2}| + |\beta_{K_{1,2}}b^{K_{1,2}}_{ij}t^{2/3} + \Tilde{x}^0_{ij}|\leq \sqrt{k'}t^{2/3}$. The denominator of the last integral is positive when 
\begin{displaymath}
    t>\bigg(\frac{|a_{ij}|^2}{k'}\bigg)^{2/3}=:\hat{T}.
\end{displaymath}
If we choose $\bar{T}(k)\gg\hat{T}$, the last integral is finite, which means that
\begin{displaymath}
    \int_{\bar{T}}^{+\infty}\bigg| \frac{1}{|\varphi^n_{ij}(t) + a_{ij}t + \beta_{K_{1,2}}b^{K_{1,2}}_{ij}t^{2/3} + \Tilde{x}^0_{ij}|} - \frac{1}{|a_{ij}t|}\bigg|\ \ud t < +\infty.
\end{displaymath}
This implies the $L^1$-convergence of the potential term, which proves its weak lower semicontinuity.

\subsection{The action is of class $C^1$ over non-collision sets}

The last thing we have to prove is that the action is of class $C^1$ over sets of motions that don't undergo collisions. We have already proved this result for the terms $\mathcal{A}_K$, so we can only focus on the terms $\mathcal{A}_{K_1,K_2}$. In particular, denoting by $\mathcal{A}_{K_1,K_2}^2$ the potential term, we wish to prove that the differential 
\begin{displaymath}
    \ud \mathcal{A}_{K_1,K_2}(\varphi)[\psi] = \int_{1}^{+\infty} \langle \nabla U(\varphi(t)+at+\beta_{K_{1,2}} b^{K_{1,2}}t^{2/3}+\Tilde{x}^0), \psi(t)\rangle\ \ud t
\end{displaymath}
is continuous, for every $\varphi,\psi\in\mathcal{D}_0^{1,2}(1,+\infty)$, over the set of non-collisional configurations when the potential $U$ is restricted to the clusters $K_1$ and $K_2$.

First of all, we have 
\begin{displaymath}
    \|\nabla U(\varphi(t)+at+\beta_{K_{1,2}} b^{K_{1,2}}t^{2/3}+\Tilde{x}^0)\|_\mathcal{M} \leq C_{16}\sum_{i\in K_1,\ j\in K_2}\frac{1}{|\varphi^n_{ij}(t) + a_{ij}t + \beta_{K_{1,2}}b^{K_{1,2}}_{ij}t^{2/3} + \Tilde{x}^0_{ij}|^2}
\end{displaymath}
for a proper constant $C_{16}$, where the right-hand side term behaves like $1/t^2$ when $t\rightarrow+\infty$. This, together with the Cauchy-Schwartz inequality, proves that the differential is well-defined.

Now, given $(\varphi^n)_n\subset\mathcal{D}_0^{1,2}(1,+\infty)$ such that $\varphi^n\rightarrow\varphi$ in $\mathcal{D}_0^{1,2}(1,+\infty)$ for some $\varphi$, we wish to prove that
\begin{displaymath}
    \sup_{\|\psi\|_\mathcal{D}\leq 1} \bigg|\int_{1}^{+\infty} \langle \nabla U(t,\varphi^n(t))-\nabla U(t,\varphi(t)),\psi(t)\rangle\ \ud t\bigg| \rightarrow0,\qquad\text{as }n\rightarrow+\infty,
\end{displaymath}
where we write $U(t,\varphi(t)):=U(\varphi(t)+at+\beta_{K_{1,2}} b^{K_{1,2}}t^{2/3}+\Tilde{x}^0)$ to lighten the notation. Using Cauchy-Schwartz and Hardy inequalities, we have
\begin{displaymath}
    \begin{split}
    &\sup_{\|\psi\|_\mathcal{D}\leq 1} \bigg|\int_{1}^{+\infty} \langle \nabla U(t,\varphi^n(t))-\nabla U(t,\varphi(t)),\psi(t)\rangle\ \ud t\bigg| \\
    &\leq \sup_{\|\psi\|_\mathcal{D}\leq 1} \int_{1}^{+\infty} t\| \nabla U(t,\varphi^n(t))-\nabla U(t,\varphi(t))\|_\mathcal{M} \frac{\|\psi(t)\|_\mathcal{M}}{t}\ \ud t\\
    &\leq \sup_{\|\psi\|_\mathcal{D}\leq 1} \bigg( \int_{1}^{+\infty} \frac{\|\psi(t)\|_\mathcal{M}^2}{t^2}\ \ud t\bigg)^{1/2} \bigg( \int_{1}^{+\infty} t^2 \| \nabla U(t,\varphi^n(t))-\nabla U(t,\varphi(t))\|_\mathcal{M}^2\ \ud t\bigg)^{1/2} \\
    & \leq 2 \bigg( \int_{1}^{+\infty} t^2 \| \nabla U(t,\varphi^n(t))-\nabla U(t,\varphi(t))\|_\mathcal{M}^2\ \ud t\bigg)^{1/2}.  
    \end{split}
\end{displaymath} 
Now, we can write
\begin{displaymath}\begin{split}
    &\int_{1}^{+\infty} t^2 \| \nabla U(t,\varphi^n(t))-\nabla U(t,\varphi(t))\|_\mathcal{M}^2\ \ud t \\
    & = \int_{1}^{+\infty} t^2 \bigg| \int_{0}^{1} \nabla^2 U(\varphi(t)+at+\beta_{K_{1,2}} b^{K_{1,2}}t^{2/3}+\Tilde{x}^0+s(\varphi^n(t)-\varphi(t)))(\varphi^n(t)-\varphi(t))\ \ud s\ \bigg|^2\ \ud t \\
    & \leq \int_{1}^{+\infty} t^2 \bigg( \int_{0}^{1} C_{16} \sum_{i\in K_1,\ j\in K_2}\frac{1}{|\varphi_{ij}(t) + a_{ij}t + \beta_{K_{1,2}}b^{K_{1,2}}_{ij}t^{2/3} + \Tilde{x}^0_{ij} + s(\varphi^n_{ij}(t)-\varphi_{ij}(t))|^3}\|\varphi^n(t)-\varphi(t)\|_\mathcal{M}\ \ud s\bigg)^2\ \ud t \\
    & \leq \int_{1}^{+\infty}  \bigg( \int_{0}^{1} C_{16} \sum_{i\in K_1,\ j\in K_2}\frac{1}{|\varphi_{ij}(t) + a_{ij}t + \beta_{K_{1,2}}b^{K_{1,2}}_{ij}t^{2/3} + \Tilde{x}^0_{ij} + s(\varphi^n_{ij}(t)-\varphi_{ij}(t))|^3}\|\varphi^n-\varphi\|_\mathcal{D} t^{3/2}\ \ud s\bigg)^2\ \ud t \\
    & \leq C_{17}\|\varphi^n-\varphi\|_\mathcal{D}
\end{split}\end{displaymath}
for a proper constant $C_{17}\in\R$, where the last term goes to zero as $n\rightarrow+\infty$. This concludes the proof.

\subsection{Absence of collisions and partial hyperbolicity of the motion}

Again, Marchal's Theorem implies that the motion we are considering has no collisions. Given
\begin{displaymath}
    x(t) = \varphi(t) + at + \beta b t^{2/3} + \Tilde{x}^0,
\end{displaymath}
we have
\begin{displaymath}
    \dot{x}(t) = \dot{\varphi}(t) + a + \frac{2}{3}\beta b t^{-1/3}.
\end{displaymath}
In this case, the conservation of the energy implies that the energy of the motion is positive.
\begin{rem}
    We observe that Chazy's Theorem can be applied to the cases of hyperbolic and hyperbolic-parabolic motions, because for completely parabolic motions the energy constant of the internal motion is null. In such cases, the limit shape of $x(t)$ is the shape of the configuration $a$ and, moreover, $L = \lim_{t\rightarrow+\infty}\frac{\max_{i<j}|x_{ij}(t)|}{\min_{i<j}|x_{ij}(t)|} < +\infty$ if and only if $x$ is hyperbolic. If the energy $h > 0$ and $L = +\infty$, then either the motion is hyperbolic-parabolic or it is not expansive.
\end{rem}
In our case, it is trivial to prove that $L=+\infty$, which implies that the motion is hyperbolic-parabolic.
\begin{rem}
    We can observe that if the indexes $i,j$ belong to the same cluster, we have $\dot{x}_{ij}(t)\rightarrow0$ when $t\rightarrow+\infty$, while if $i,j$ belong to different clusters, we have $\dot{x}_{ij}(t)\rightarrow a_{ij}$ when $t\rightarrow+\infty$.
\end{rem}

\subsection{Hyperbolic-parabolic motions' asymtptotic expansion}

We have seen that a hyperbolic-parabolic motion $x$ can be written in the form $x(t) = at + \beta bt^{2/3} + \varphi(t) + \Tilde{x}^0$, as shown in \eqref{explicit_hyperbolic-parabolic}, and that the bodies can be divided into subgroups following the natural cluster partition introduced in Definition \ref{def_cluster_partition}. In this section, we will prove that the centers of mass of the clusters follow hyperbolic orbits. Besides, we will show that inside each cluster, the bodies move with respect to the center of mass of the cluster following a parabolic path.

We start with proving that the centers of mass of each cluster have a hyperbolic expansion. For a cluster $K$, denoting the center of mass of $K$ as
\begin{displaymath}
    c^K(t) = \frac{1}{M_K}\sum_{i\in K}m_i x_i(t),
\end{displaymath}
we can compute the equations of motion of the center of mass as
\begin{displaymath}
    \begin{split}
        M_K \ddot{c}^K(t) & = \sum_{i\in K} m_i \ddot{x}_i(t) \\
        & = -\sum_{i\in K}\sum_{j\neq i} m_i m_j \frac{x_i(t) - x_j(t)}{|x_i(t) - x_j(t)|^3} \\
        & = -\sum_{i\in K}\sum_{j\notin K} m_i m_j \frac{x_i(t) - x_j(t)}{|x_i(t) - x_j(t)|^3}.
    \end{split}
\end{displaymath}
It is easy to see that the right-hand side of the equation is a $O\big(\frac{1}{t^2}\big)$-term for $t\rightarrow+\infty$. We also notice that
\begin{displaymath}
    -\sum_{i\in K}\sum_{j\notin K} m_i m_j \frac{x_i(t) - x_j(t)}{|x_i(t) - x_j(t)|^3} \simeq -\frac{1}{t^2}\sum_{i\in K}\sum_{j\notin K} m_i m_j \frac{a_i - a_j}{|a_i - a_j|^3} + O\bigg(\frac{1}{t^3}\bigg),
\end{displaymath}
for $t\rightarrow+\infty$. We can define 
\begin{displaymath}
    \Tilde{\nabla}U(a^K) = -\sum_{i\in K}\sum_{j\notin K} m_i m_j \frac{a_i - a_j}{|a_i - a_j|^3},
\end{displaymath}
which can be seen as a restriction of $\nabla U(a^K)$. Denoting with $a^K$ the restriction of the configuration $a$ to the cluster $K$, we can thus compute
\begin{displaymath}
    \lim_{t\rightarrow+\infty} \frac{M_K c^K(t)}{\log t} = \lim_{t\rightarrow+\infty} \frac{M_K \dot{c}^K(t)}{\frac{1}{t}} = -\lim_{t\rightarrow+\infty} \frac{M_K \ddot{c}^K(t)}{\frac{1}{t^2}} = -\Tilde{\nabla}U(a^K).
\end{displaymath}
This implies that the center of mass of the cluster $K$ has the hyperbolic asymptotic expansion
\begin{displaymath}
    c^K(t) = a^K t - \Tilde{\nabla}U(a^K)\log t + o(\log t),
\end{displaymath}
for $t\rightarrow +\infty$.

Now, considering an index $i\in K$, we denote the motion of a body $x_i$ with respect to the center of mass of its cluster as
\begin{displaymath}
    y_i(t) = x_i(t) - c^K_i(t).
\end{displaymath}
We are going to show that its asymptotic expansion is a parabolic one. 

If the cluster only has one element, we obviously have $y_i\equiv0$, so we consider the case where $K$ has two or more elements. The equation of motion reads
\begin{displaymath}
    \begin{split}
        m_i \Ddot{y}_i(t) & = m_i \Ddot{x}_i(t) - m_i\Ddot{c}^K_i(t) \\
        & = -\sum_{j\in K} m_i m_j \frac{x_i(t) - x_j(t)}{|x_i(t) - x_j(t)|^3} - \sum_{j\notin K} m_i m_j \frac{x_i(t) - x_j(t)}{|x_i(t) - x_j(t)|^3} - m_i \Ddot{c}^K_i(t).
    \end{split}
\end{displaymath}
Since we already know that $-\sum_{j\notin K} m_i m_j \frac{x_i(t) - x_j(t)}{|x_i(t) - x_j(t)|^3} - m_i \Ddot{c}^K_i(t) = O\big(\frac{1}{t^2}\big)$ for $t\rightarrow+\infty$, we can then say that
\begin{displaymath}
    m_i \Ddot{y}_i(t) = -\sum_{j\in K} m_i m_j \frac{y_i(t) - y_j(t)}{|y_i(t) - y_j(t)|^3} + O\big(\frac{1}{t^2}\big).
\end{displaymath}

Using the definition of $x(t)$ and the asymptotic expansion of $c^K(t)$ we found above, we can easily see that
\begin{displaymath}
    y_i(t) = \beta_K b^K_i t^{2/3} + \varphi_i(t) - \log t\sum_{j\notin K} m_i m_j \frac{a_i - a_j}{|a_i - a_j|^3} + o(\log t),
\end{displaymath}
for $t\rightarrow+\infty$, where $\beta_K = \sqrt[3]{\frac{9}{2}\min_{K} U}$. Defining $\psi_i(t):=\varphi_i(t) - \sum_{j\notin K} m_i m_j \frac{a_i - a_j}{|a_i - a_j|^3} + o(\log t)$, it is easy to prove that $\psi_i\in\mathcal{D}^{1,2}(1,+\infty)$. We can then apply the estimate \eqref{new_estimate_parabolic} to say that
\begin{displaymath}
    y_i(t) = \beta_K b_i^K t^{2/3} + o(t^{\frac{1}{3}+}),
\end{displaymath}
for $t\rightarrow+\infty$.

\section{Free-time minimization property}

``Jacobi's principle brings out vividly the intimate relationship which exists between the motions of conservative holonomic systems and the geometry of curved space'' (C. Lanczos, \cite[page 138]{Lanczos}). Accordingly, trajectories of the $N$-body problem at energy $h$ are geodesics of the Jacobi-Maupertuis' metric of level $h$ in the configuration space, i.e.,
\[\ud\sigma^2=(U+h)\ud s_\mathcal{M}^2,
\]
being $\ud s_\mathcal{M}^2$ the mass Euclidean metric in the configuration space. 
\begin{defn}
    A curve $x:[1, +\infty) \rightarrow E^N$ is said to be a geodesic ray from $p \in E^N$ if $x(1) = p$ and each restriction to a compact interval is a minimizing geodesic.
\end{defn}
In \cite{MadernaVenturelli_HyperbolicMotions}, Maderna and Venturelli also proved the following theorem.

\begin{thm}[Maderna-Venturelli, 2020 \cite{MadernaVenturelli_HyperbolicMotions}]\label{thm_MV_geodesicrays}
    Let $E$ be an Euclidean space. For any $h>0$, $p \in E^N$ and $a \in \Omega$, there is geodesic ray of the Jacobi-Maupertuis' metric of level $h$ with asymptotic direction $a$ and starting at $p$.
\end{thm}
In order to relate geodesics of the Jacobi-Maupertuis' metric to the action minimizing trajectories of our Lagrangian systems we need the following definition.
\begin{defn}
    A curve $x:I\rightarrow\mathcal{X}$ is a free-time minimizer  for the Lagrangian action at energy $h$ if $\forall\ [a,b],[a',b']\subset I$ and $\forall\ \sigma:[a',b']\rightarrow\mathcal{X}$ such that $\gamma(a)=\sigma(a')$ and $\gamma(b)=\sigma(b')$, it holds 
    \begin{displaymath}
        \int_{a}^{b} L(\gamma,\dot{\gamma})\ \ud t+h(b-a) \leq \int_{a'}^{b'} L(\sigma,\dot{\sigma})\ \ud t+h(b'-a').
    \end{displaymath}
\end{defn}
In light of the equivalence between the variational property of being an unbouded free-time minimizer of the Lagrangian action at energy $h$ and the geometrical property of being a geodesic ray for the Jacobi-Mapertuis metric at the same energy level (cfr. \cite{Lanczos,MR2269239}), we show here that our existence results of expansive motions through the minimization of the renormalized action do indeed agree with Theorem \ref{thm_MV_geodesicrays}. More precisely, we prove the following corollary.
\begin{cor}
    Consider an expansive motion $x:[1,+\infty)\rightarrow\mathcal{X}$ of the Newtonian $N$-body problem of the form 
    \begin{displaymath}
        x(t)=r_0(t) + \varphi(t)+\Tilde x_0,
    \end{displaymath}
    where $\varphi\in\mathcal{D}_0^{1,2}(1,+\infty)$ minimizes the renormalized action in \eqref{eq:renorm_action} in any of the settings of Theorems \ref{thm_hyperbolic}, \ref{thm_parabolic} and \ref{thm_partially_hyperbolic}. Then $x$ is actually a free-time minimizer at its energy level. Therefore it is a geodesic ray for the Jacobi-Maupertuis' metric.
\end{cor}

\begin{proof}
We consider a curve $\gamma:[1,+\infty)\rightarrow\mathcal{X}$ of the form $\gamma(t) = r_0(t) + \varphi(t) + \Tilde{x}^0$ such that $\varphi$ minimizes the renormalized Lagrangian action on $\mathcal{D}_0^{1,2}(1,+\infty)$. 

By contradiction, we suppose that there are $T$ and $\bar{T}$, $\varepsilon>0$ and there is some curve $\bar{\sigma}:[1,\bar{T}]\rightarrow\mathcal{X}$ with $\gamma(T)=\bar{\sigma}(\bar{T})$ such that
    \begin{equation}\label{eq_freetime_1}
        \int_{1}^{T} L(\gamma,\dot{\gamma})\ \ud t +hT> \int_{1}^{\bar{T}} L(\bar{\sigma},\dot{\bar{\sigma}})\ \ud t+h\bar T + \varepsilon.
    \end{equation}

By a density and continuity argument, we can then define a compactly supported function $\Tilde{\varphi}$ such that $\Tilde{\varphi}(t)=\varphi(t)$ on $[1,\hat{T}]$, where $\hat{T}\gg \max\{T,\bar{T}\}$, and $\Tilde{\varphi}$ is close enough to $\varphi$ in the $\mathcal{D}_0^{1,2}$-norm to have
\begin{displaymath}
    \mathcal{A}(\Tilde{\varphi}) \leq \mathcal{A}(\varphi) + \varepsilon,
\end{displaymath}
where $\mathcal{A}$ is the renormalized Lagrangian action. By the minimizing property of $\varphi$ we infer
\begin{equation}\label{eq:minepsilon}
    \mathcal{A}(\Tilde{\varphi}) \leq \mathcal{A}(\psi) + \varepsilon,\quad\forall \psi\in\mathcal D^{1,2}_0([1,+\infty)).
\end{equation}

Now, denoting $\Tilde{\gamma}(t) = r_0(t) + \Tilde{\varphi}(t) + \Tilde{x}^0$, we build a curve $\tilde{\sigma}:[1,+\infty)\rightarrow\mathcal{X}$ such that
    \begin{displaymath}
        \tilde{\sigma}(t) = \begin{cases}
            \bar{\sigma}(t),\quad t\in[1,\bar{T}]\\
            \tilde{\gamma}(t-\bar{T}+T),\quad t\in[\bar{T},+\infty)
        \end{cases}.
    \end{displaymath}
    Since we supposed that $\gamma(T)=\bar{\sigma}(\bar{T})$, we know for sure that $\tilde{\sigma}$ is continuous. Moreover we define $\bar\varphi(t)=\bar\sigma(t)-r_0(t)-\tilde x_0$, so that $\bar\varphi\in\mathcal D^{1,2}_0(1,+\infty)$ and, by its definition, we have
    \begin{equation}\label{eq:barphi}
    	\bar\varphi(t)\equiv r_0(t-\bar T+T)-r_0(t)=a(T-\bar T)+o(1), \quad\forall t\gg\max\{T,\bar T\},
	\end{equation}
as $\Tilde{\varphi}$ is compactly supported. We notice that we can write
\begin{displaymath}
    \mathcal{A}(\tilde{\varphi}) = \int_{1}^{+\infty} L(\Tilde{\gamma},\dot{\Tilde{\gamma}}) - L_0(t)\ \ud t,
\end{displaymath}
which easily follows from the fact that also $L(\Tilde{\gamma},\dot{\tilde{\gamma}}) - L_0(t)\in L^1[1,+\infty)$ and furthermore
\begin{displaymath}
    \int_{1}^{+\infty} -\langle \mathcal{M}\ddot{r_0},\tilde\varphi\rangle\ \ud t = -\langle \mathcal{M}\dot{r_0},\tilde\varphi\rangle\bigg|_1^{+\infty} + \int_{1}^{+\infty} \langle \mathcal{M}\dot{r_0},\dot{\tilde\varphi}\rangle\ \ud t = \int_{1}^{+\infty} \langle \mathcal{M}\dot{r_0},\dot{\tilde\varphi}\rangle\ \ud t.
\end{displaymath}
On the other hand, from \eqref{eq:barphi}, using $\dot r_0\simeq t^{-1/3}$, it follows  that 
\begin{displaymath}
\begin{split}
    \int_{1}^{+\infty} -\langle \mathcal{M}\ddot{r_0},\bar\varphi\rangle\ \ud t = -\langle \mathcal{M}\dot{r_0},\bar\varphi\rangle\bigg|_1^{+\infty} + \int_{1}^{+\infty} \langle \mathcal{M}\dot{r_0},\dot{\bar\varphi}\rangle\ \ud t= \langle \mathcal{M}a,a \rangle (\bar T-T)+\int_{1}^{+\infty} \langle \mathcal{M}\dot{r_0},\dot{\bar\varphi}\rangle\ \ud t \\= 2h (\bar T-T)+\int_{1}^{+\infty} \langle \mathcal{M}\dot{r_0},\dot{\bar\varphi}\rangle\ \ud t,
\end{split}
\end{displaymath}
where $h=H(r_0,\dot{r}_0)$ is the energy of $r_0$, which is positive equal to $\Vert a\Vert^2_\mathcal{M}/2$ in the hyperbolic and hyperbolic-parabolic case and zero in the completely parabolic case.
Consequently we have
\begin{displaymath}
    \mathcal{A}(\bar{\varphi}) = 2h (\bar T-T)+\int_{1}^{+\infty} L(\bar{\sigma},\dot{\bar{\sigma}}) - L_0(t)\ \ud t.
\end{displaymath}

Let us denote  $L^h=L-h$ and $L_0^h(t):=L(r_0(t))-h$. By \eqref{eq_freetime_1}, we can say that
\begin{equation}\label{eq_freetime_2}
 \begin{split}
   & \int_{1}^{T} L^h(\gamma,\dot{\gamma})\ \ud t + \int_{\bar{T}}^{+\infty} L^h(\tilde{\sigma},\dot{\tilde{\sigma}})-L^h_0(t-\bar{T}+T)\ \ud t \\
   & > \int_{1}^{\bar{T}} L^h(\bar{\sigma},\dot{\bar{\sigma}})\ \ud t  + \int_{\bar{T}}^{+\infty} L^h(\tilde{\sigma},\dot{\tilde{\sigma}})-L^h_0(t-\bar{T}+T)\ \ud t + \varepsilon+2h(\bar T-T).
\end{split}
\end{equation}
Working on left-hand side of equation \eqref{eq_freetime_2}, we obtain
\begin{displaymath}\begin{split}
    & \int_{1}^{T} L^h(\gamma,\dot{\gamma})\ \ud t + \int_{\bar{T}}^{+\infty} L^h(\tilde{\sigma},\dot{\tilde{\sigma}})-L^h_0(t-\bar{T}+T)\ \ud t \\
    & = \int_{1}^{T} L^h(\gamma,\dot{\gamma})\ \ud t + \int_{\bar{T}}^{+\infty} L^h(\Tilde{\gamma}(t-\bar{T}+T),\dot{\Tilde{\gamma}}(t-\bar{T}+T))-L^h_0(t-\bar{T}+T)\ \ud t\\
    & = \int_{1}^{T} L^h(\gamma,\dot{\gamma}) - L^h_0(t)\ \ud t + \int_{T}^{+\infty} L^h(\Tilde{\gamma},\dot{\Tilde{\gamma}}) - L^h_0(t)\ \ud t + \int_{1}^{T}L^h_0(t)\ \ud t\\
    & = \int_{1}^{+\infty} L^h(\Tilde{\gamma},\dot{\Tilde{\gamma}}) - L^h_0(t)\ \ud t + \int_{1}^{T}L^h_0(t)\ \ud t.
\end{split}\end{displaymath}
On the other hand, working on right-hand side of \eqref{eq_freetime_2}, we have
\begin{displaymath}\begin{split}
    & \int_{1}^{\bar{T}} L^h(\bar{\sigma},\dot{\bar{\sigma}})\ \ud t  + \int_{\bar{T}}^{+\infty} L^h(\tilde{\sigma},\dot{\tilde{\sigma}})-L^h_0(t-\bar{T}+T)\ \ud t +2h(\bar T-T)+ \varepsilon \\
    & = \int_{1}^{\bar{T}} L^h(\bar{\sigma},\dot{\bar{\sigma}}) - L^h_0(t)\ \ud t  + 
    \int_{\bar{T}}^{+\infty} L^h(\tilde{\sigma},\dot{\tilde{\sigma}})-L^h_0(t-\bar{T}+T) + L^h_0(t) - L^h_0(t)\ \ud t +
    \\
    &\hspace{9cm}  +\int_{1}^{\bar{T}} L^h_0(t)\ \ud t +2h(\bar T-T)+ \varepsilon \\
    & = \int_{1}^{\bar{T}} L^h(\bar{\sigma},\dot{\bar{\sigma}}) - L^h_0(t)\ \ud t  + \int_{\bar{T}}^{+\infty} L^h(\tilde{\sigma},\dot{\tilde{\sigma}})-L^h_0(t)\ \ud t+ \\
    &\hspace{4cm} + \int_{1}^{\bar{T}} L^h_0(t)\ \ud t  + \int_{\bar{T}}^{+\infty} L^h_0(t) - L^h_0(t-\bar{T}+T)\ \ud t +2h(\bar T-T)+ \varepsilon \\
    & = \int_{1}^{+\infty} L^h(\tilde{\sigma},\dot{\tilde{\sigma}})-L^h_0(t)\ \ud t + \int_{1}^{\bar{T}} L^h_0(t)\ \ud t  + \int_{\bar{T}}^{+\infty} L^h_0(t) - L^h_0(t-\bar{T}+T)\ \ud t +2h(\bar T-T)+ \varepsilon.
\end{split}\end{displaymath}
It thus follows that
\begin{displaymath}
\begin{split}
    & \int_{1}^{+\infty} L^h(\Tilde{\gamma},\dot{\Tilde{\gamma}}) - L^h_0(t)\ \ud t \\
    &> \int_{1}^{+\infty} L^h(\tilde{\sigma},\dot{\tilde{\sigma}})-L^h_0(t)\ \ud t + \int_{T}^{\bar{T}} L^h_0(t)\ \ud t  + \int_{\bar{T}}^{+\infty} L^h_0(t) - L^h_0(t-\bar{T}+T)\ \ud t + 2h(\bar T-T)+\varepsilon.
    \end{split}
\end{displaymath}
We recall the following property, which can be demonstrated as a simple exercise.
\begin{prop}\label{lem_freetime}
    Given a function $f\in L^1_{loc}(\mathcal{X})$ such that $f(t)\rightarrow0$ as $t\rightarrow\pm\infty$ and such that $f(t)-f(t-\tau)\in L^1$ for some $\tau\in\R$, then
    \begin{displaymath}
        \int_{-\infty}^{+\infty} f(t) - f(t-\tau)\ \ud t = 0.
    \end{displaymath}
\end{prop}
Since
\begin{displaymath}
    \int_{T}^{\bar{T}} L^h_0(t)\ \ud t  + \int_{\bar{T}}^{+\infty} L^h_0(t) - L^h_0(t-\bar{T}+T)\ \ud t = \int_{-\infty}^{+\infty} L^h_0(t)\mathcal{X}_{\{t>T\}} - L^h_0(t-\bar{T}+T)\mathcal{X}_{\{t>\bar{T}\}}\ \ud t,
\end{displaymath}
we can apply the Proposition \ref{lem_freetime} to the function $L^h_0(t)\mathcal{X}_{\{t>T\}}$. This eventually yields
\begin{displaymath}
 \int_{1}^{+\infty} L^h(\Tilde{\gamma},\dot{\Tilde{\gamma}}) - L^h_0(t)\ \ud t \\
    > \int_{1}^{+\infty} L^h(\tilde{\sigma},\dot{\tilde{\sigma}})-L^h_0(t)\ \ud t + 2h(\bar T-T)+\varepsilon, 
 \end{displaymath}   
    and finally
    \[
    \mathcal A(\Tilde\varphi)>\mathcal A(\bar\varphi)+\varepsilon,
    \]
in clear contradiction with \eqref{eq:minepsilon}.
\end{proof}

\section{Hamilton-Jacobi equations}\label{sec:HJ}

We now emphasize the dependence on the initial point $x^0$ and define the function
\begin{equation}
\begin{split}\label{eq:sol_HJ}
   v(x_0)&=\min_{\varphi \in \mathcal D^{1,2}_0(1,+\infty)} \left\{ \int_{1}^{+\infty} \frac{1}{2}\|\dot{\varphi}(t)\|_\mathcal{M}^2 \ \ud t\right.+
  \\&+ \left. \int_{1}^{+\infty} \ U(\varphi(t) + r_0(t) + x^0-r_0(1)) - U(r_0(t)) - \langle \ddot{r}_0(t),\varphi(t)\rangle_\mathcal{M}\ \ud t\right\}- \langle a, x^0\rangle_\mathcal{M}.
    \end{split}
    \end{equation}
    
    We claim that $u$ solves the Hamilton-Jacobi equation
    \begin{equation}\label{eq:HJ}
    H(x,\nabla v(x))=h
    \end{equation}
    in the viscosity sense. This can be easily seen by taking a point $x^0$ of differentiability, and formally differentiate \eqref{eq:sol_HJ} with respect to $x^0$, finding 
\[
    \nabla v(x^0)=-\mathcal{M}\dot x(1)
\]
where $x(t)=r_0(t)+\varphi(t)+x^0-r_0(1)$ and $\varphi$ is the minimizer of the renormalized action associated with $x^0$. Therefore $\mathcal{M}^{-1/2}\nabla v(x^0)=-\mathcal{M}^{1/2}\dot x(1)$ and we easily obtain \eqref{eq:HJ} from the expression of the Hamiltonian \eqref{eq:hamiltonian}. Making this argument fully rigorous goes beyond the scope of this paper. The interested reader can retrace step by step the method explained in \cite{MadernaVenturelli_HyperbolicMotions}, also taking into account that it is known that the singular set is contained in a locally countable union of smooth hypersurfaces of codimension at least one (cfr. \cite{MR2041617}).

Fixing $x^0$ and $T>0$, we now consider the boundary value problem 
\begin{displaymath}
    \begin{cases}
        \mathcal{M}\ddot{x} = \nabla U(x)\\
        x(1)=x^0\\
        \dot{x}(T)=\dot{r}_0(T)
    \end{cases}
\end{displaymath}
and introduce the associated value function
\begin{displaymath}
    u(T,x^0) = \min_{\gamma\in H^1([1,T]),\ \gamma(1)=x^0} \int_{1}^{T}  \frac{1}{2}\|\dot{\gamma}(t)\|_\mathcal{M}^2 + U(\gamma(t))\ \ud t - \langle\dot{r}_0(T),\gamma(T)\rangle_\mathcal{M}.
\end{displaymath}
It is a standard result of the theory of Hamilton-Jacobi equations (cfr. \cite{MR2041617}) that $u$ is a viscosity solution of
\begin{displaymath}
    -\frac{\partial u}{\partial T} = \frac{1}{2}\|\nabla u\|^2_{\mathcal{M}^{-1}} - U(x),
\end{displaymath}
where the gradient is taken with respect to the second variable.
\begin{rem}
    Notice that, compared with \cite{MR2041617}, we have reversed time orientation. 
\end{rem}

Now, we define
\begin{displaymath}
    v(T,x) = u(T,x) + \int_{1}^{T}\frac{1}{2}\|\dot{r}_0(t)\|_\mathcal{M}^2 - U(r_0(t))\ \ud t = u(T,x) + \int_{1}^{T} H(r_0(t),\dot{r}_0(t))\ \ud t
\end{displaymath} 
and observe that
\begin{displaymath}
    -\frac{\partial v}{\partial T} = \frac{1}{2}\|\nabla v\|_{\mathcal{M}^{-1}} - U(x) - H(r_0,\dot{r}_0).
\end{displaymath}
Assume that $v(T,x)$ converges uniformly to some $v(x)$ as $T\rightarrow+\infty$. Then, $v$ is a stationary viscosity solution to the stationary Hamilton-Jacobi equation
\begin{displaymath}
    \frac{1}{2}\|\nabla v\|_{\mathcal{M}^{-1}} - U(x) = \lim_{T\rightarrow+\infty} H(r_0,\dot{r}_0) = \frac{1}{2}\|a\|_\mathcal{M}^2.
\end{displaymath}
To relate the modified value function $v$ with the minimum of our renormalized action, let us write
\begin{displaymath}
    \gamma(t) = r_0(t) + \varphi(t) + \Tilde{x}^0,
\end{displaymath}
with $\Tilde{x}^0=x^0-r_0(1)$, and  compute
\begin{displaymath}
\begin{split}
    &\int_{1}^{T} \frac{1}{2}\|\dot{r}_0(t) + \dot{\varphi}(t)\|_\mathcal{M}^2 + U(r_0(t)+\varphi(t)+\Tilde{x}^0) + \frac{1}{2}\|\dot{r}_0(t)\|_\mathcal{M}^2 - U(r_0(t))\ \ud t - \langle\dot{r}_0(T),r_0(T)+\varphi(T)+x^0-r_0(1)\rangle_\mathcal{M}\\
    & = \int_{1}^{T} \frac{1}{2}\|\dot{\varphi}(t)\|_\mathcal{M}^2 + U(r_0(t)+\varphi(t)+\Tilde{x}^0) - U(r_0(t)) - \langle\ddot{r}_0(t),\varphi(t)\rangle_\mathcal{M}\ \ud t - \langle\dot{r}_0(T), x^0\rangle_\mathcal{M},
\end{split}\end{displaymath}
which follows from some integration by parts. Therefore, we have
\begin{displaymath}
    v(T,x^0) = \min_{\varphi\in H^1([1,T]),\ \varphi(1)=0} \mathcal{A}_{[1,T]}^{ren}(\varphi) - \langle\dot{r}_0(T), x^0\rangle_\mathcal{M},
\end{displaymath}
where we denoted 
\begin{displaymath}
    \mathcal{A}_{[1,T]}^{ren}(\varphi) = \int_{1}^{T} \frac{1}{2}\|\dot{\varphi}(t)\|_\mathcal{M}^2 + U(r_0(t)+\varphi(t)+\Tilde{x}^0) - U(r_0(t)) - \langle\ddot{r}_0(t),\varphi(t)\rangle_\mathcal{M}\ \ud t.
\end{displaymath}
Then, it becomes natural to let $T\to +\infty$ and define
\begin{displaymath}
    v(x^0) = \min_{\varphi\in\mathcal{D}_0^{1,2}(1,+\infty)}\int_{1}^{+\infty} \frac{1}{2}\|\dot{\varphi}(t)\|_\mathcal{M}^2 + U(r_0(t) + \varphi(t) + x^0 - r_0(1)) - U(r_0(t)) - \langle\ddot{r}_0(t),\varphi(t)\rangle_\mathcal{M}\ \ud t - \langle a, x^0\rangle_\mathcal{M}.
\end{displaymath}
We will prove in a forthcoming paper that 
\begin{displaymath}
    v(x) = \lim_{T\rightarrow+\infty} v(T,x)
\end{displaymath}
uniformly on compact sets of $\R^{dN}$ (actually, in the Hölder norms), so that $v$ solves
\begin{displaymath}
    \frac{1}{2}\|\nabla v\|_{\mathcal{M}^{-1}} - U(x) = \frac{1}{2}\|a\|_\mathcal{M}^2
\end{displaymath}
in the viscosity sense. This justifies once again our choice for the renormalized action functional. 

It is worthwhile noticing that the uniqueness result in \cite{MadernaVenturelli_inpreparation} ensures that, in the hyperbolic case, our value function $v$  is indeed the Busemann function. Moreover, it may be interesting that the linear correction in \eqref{eq:sol_HJ} is itself the Busemann function of the free particle.



\bibliographystyle{siam}

\end{document}